\documentclass[12pt]{article}
\usepackage[titletoc,title]{appendix}
\usepackage{amsmath,amsfonts,amssymb,bm,amsthm,empheq}
\usepackage[colorlinks,citecolor=blue,linkcolor=blue]{hyperref}
\usepackage{mathrsfs}

\usepackage{tikz}
\usetikzlibrary{shapes,arrows}
\usetikzlibrary{intersections,shapes.arrows}
\usetikzlibrary{decorations.pathreplacing}
\usetikzlibrary{arrows.meta}
\numberwithin{equation}{section}
\newtheorem{theorem}{Theorem}[section]
\newtheorem{lemma}[theorem]{Lemma}
\newtheorem{proposition}[theorem]{Proposition}
\newtheorem{corollary}[theorem]{Corollary}
\theoremstyle{definition}

\newtheorem{definition}[theorem]{Definition}
\newtheorem{assumption}[theorem]{Assumption}
\theoremstyle{remark}
\newtheorem{remark}[theorem]{Remark}
\newtheorem{remarks}[theorem]{Remarks}
\newtheorem{notation}[theorem]{Notation}

\usepackage[T1]{fontenc}
\usepackage[margin=2.41cm]{geometry}
\usepackage{titlesec}
\titleformat{\section}{\large\bfseries\filcenter}{\thesection}{1em}{}
\titleformat{\subsection}{\bfseries}{\thesubsection}{1em}{}

\usepackage{fancyhdr}
\usepackage{enumerate}
\usepackage{xcolor,cancel}
\usepackage{relsize}

\usepackage{tikz}
\usetikzlibrary{matrix,arrows,decorations.pathmorphing}
\input{xy}
\xyoption{all}

\usepackage{cite}

\usepackage{hyperref}
\hypersetup{
pdftitle={},%
pdfauthor={},%
pdfsubject={},%
pdfkeywords={},%
colorlinks=true,%
linkcolor=black,%
urlcolor=black,
linktoc={},
linktocpage={false},%
pageanchor={},
citecolor={black},
}

\renewcommand{\epsilon}{{\varepsilon}}

\usepackage{bbm}

\newcommand{\resp}{\textit{resp. }}
\newcommand{\ie}{\textit{i.e.~}}
\newcommand{\eg}{\textit{e.g.~}}
\newcommand{\cf}{\textit{cf. }}

\newcommand{\RR}{{\mathbb{R}}}
\newcommand{\CC}{{\mathbb{C}}}

\renewcommand{\hat}{\widehat}

\newcommand{\Dir}{{\mathsf{D}}}

\newcommand{\G}{\mathsf{G}}
\newcommand{\Id}{{\operatorname{Id}}}
\newcommand{\sol}{\mathsf{Sol}\,}

\newcommand{\E}{\mathsf{E}}

\newcommand{\M}{\mathsf{M}}

\renewcommand{\P}{\mathsf{P}}

\renewcommand{\S}{\mathbb{S}}
\newcommand{\T}{\mathsf{T}}

\newcommand{\f}{{\mathfrak{f}}}

\renewcommand{\aa}{{\mathfrak{a}}}

\newcommand{\scrH}{\mathscr{H}}

\newcommand{\vol}{{\textnormal{vol}\,}}

\newcommand{\supp}{{\textnormal{supp\,}}}
\newcommand{\Spin}{\textnormal{Spin}}
\newcommand{\define}{:=}

\newcommand{\rquot}[2]{\raisebox{0.5ex}{$#1$}\!/\!\raisebox{-0.5ex}{$#2$}}

\newcommand{\fiber}[2]{\prec  #1\,|\, #2  \succ}

\newcommand{\scalar}[2]{(#1\, |\, #2)}

\renewcommand{\tilde}{\widetilde}

\allowdisplaybreaks

\begin{document}
%
%
%
%
\thispagestyle{empty}
\begin{center}

\vspace{5mm}

{\LARGE\bf  Global and microlocal aspects  of Dirac operators:\\[4mm] propagators and Hadamard states} 

\vspace{5mm}

{\bf by}

\vspace{5mm}
 { \bf Matteo Capoferri}\\[1mm]
\noindent  
{\it Maxwell Institute for Mathematical Sciences, Edinburgh}\\[1mm]
{\it Department of Mathematics, Heriot-Watt University, Edinburgh EH14, UKK}\\[1mm]
email: \ {\tt M.Capoferri@hw.ac.uk}\;\;
webpage:\ \url{https://mcapoferri.com}
\\[6mm]
{  \bf Simone Murro}\\[1mm]
\noindent {\it Dipartimento di Matematica, Università di Genova and INFN, sezione di Genova, Italy}\\[1mm]
email: \ {\tt  murro@dima.unige.it} \;\; webpage:\ \url{https://www.simonemurro.eu}
\\[10mm]
\end{center}

\begin{abstract}
We propose a geometric approach to construct 
the Cauchy evolution operator for the Lorentzian 
Dirac operator on Cauchy-compact globally 
hyperbolic 4-manifolds. 
We realise the Cauchy evolution operator as 
the sum of two invariantly defined oscillatory 
integrals, the positive and negative Dirac propagators,
 global in space and in time, 
 with distinguished complex-valued geometric 
 phase functions. As applications, 
 we relate the Cauchy evolution operators 
 with the Feynman propagator and construct 
 Cauchy surfaces covariances of quasifree Hadamard states.
\end{abstract}

\\[0.5mm]

\tableofcontents

\section{Introduction and main results}

Since the appearance of the seminal paper of Duistermaat and H\"ormander \cite{DuHo},  it became increasingly clear that microlocal analysis is an indispensable tool to formulate quantum field theories (QFTs) in a mathematically rigorous fashion. 
On the one hand,  it supplies powerful techniques for solving the partial differential equations that rule the dynamics of quantum fields; on the other hand,  it provides workable criteria for the existence of product of quantum observables,  viewed as a Hilbert-space-valued distributions, thus paving the way for a rigorous formulation of Wick polynomials.

Radzikowski's reformulation \cite{Radzikowski-96} of the Hadamard condition --- which essentially selects as physical states in curved spacetime those states whose ultraviolet behaviour mimics the short-scale divergence of the Poincar\'e vacuum --- in terms of the wave-front set brought microlocal analysis into a centre stage position in the algebraic formulation of quantum field theory in curved spacetime (AQFT).
Since then, microlocal analysis has been a key ingredient in numerous advancements in the field.  Remarkably,  G\'erard and Wrochna \cite{gerard0} (see also the subsequence papers~\cite{gerardHaw,gerardUnruh,gerard1,
gerardDirac,gerardYM,gerard2,
gerard3,gerard4,gerard5,wrochna0,wrochna1,wrochna2}) showed that Hadamard states can be constructed via pseudodifferential techniques on generic curved spacetimes for scalar, Dirac and even (linear) Yang-Mills fields. 

For a detailed self-contained overview of the interplay between microlocal analysis and  quantum field theory we refer the reader to the monograph~\cite{GerardBook},  which is also an excellent introduction to linear scalar QFT in curved spacetime.

The goal of this paper is to construct Cauchy evolution operators,  Feynman propagator and Hadamard states for the Dirac equation on Lorentzian manifolds using Fourier integral --- as opposed to pseudodifferential --- operator techniques. 
 This will be done in an explicit invariant fashion,  in terms of oscillatory integrals with distinguished geometric complex-valued phase functions,  global in space and time. Our approach is ``direct'', in that it does not require one to manipulate resolvents (and, hence,  parameter-dependent pseudodifferential calculi) or resort to complex analysis.

Lorentzian propagators,  on top of being of interest in their own right,  can also be viewed as an important step towards developing interacting quantum field theories in curved spacetime.  Indeed, they play a crucial role in the construction of locally covariant
Wick powers~\cite{HW} as well as in the definition of star products in formal deformation quantisation of classical field theory.

\subsection{Cauchy evolution operator}

Numerous techniques are available in the literature to construct Cauchy evolution operators for hyperbolic problems, \ie operators mapping initial data to propagating solutions.

In the relatively simple case of a scalar problem of the form
\[
(-i\partial_t+A)f=0,  \qquad f|_{t=0}=0,
\]
where $A$ is a self-adjoint elliptic pseudodifferential operator acting on a closed Riemannian manifold $M$ and time $t$ is an external parameter, the \emph{propagator}
\begin{equation}
\label{propagator riemannian}
U(t):=e^{-itA}
\end{equation}
can be written down \emph{exactly} in terms of the eigenvalues and the eigenfunctions of $A$.  Microlocal analysis offers a way of writing down $U(t)$ \emph{approximately}, modulo an integral operator with infinitely smooth kernel,  in the form of a Fourier Integral operator. 
 The classical parametrix construction,  originally proposed by Lax~\cite{lax} and Ludwig~\cite{lud},  and which can be found for example in H\"ormander's four-volume monograph \cite{hormander},  has a number of issues: (i) it is not invariant under changes of local coordinates, (ii) it is local in space and (iii) it is local in time.  The latter issue, locality in time, is especially serious, in that it is related with obstructions due to caustics. 

Laptev, Safarov and Vassiliev \cite{LSV} showed that one can achieve globality in time by writing the evolution operator as an oscillatory integral with \emph{complex-valued} phase function.  For a special class of operators,  which include the wave and Dirac equations,  one can construct the evolution operator in an explicit,  global (in space and time) and invariant fashion by identifying \emph{distinguished} complex-valued phase functions,  dictated by the geometry of the underlying space \cite{wave, dirac, part2}.

Extending Riemannian results to the Lorentzian setting, when such results continue to be true upon appropriate adaptations,  is often a non-trivial enterprise, see \eg~ \cite{strohmaier_zelditch}. This applies to evolution operators as well: one does no longer have propagators of the form \eqref{propagator riemannian},  because time $t$ is not an external parameter and $A$ will in general be time-dependent.
In~ \cite{lorentzian} the authors constructed global parametrices for the scalar wave equation on globally hyperbolic spacetimes.  However,  it is well known that switching from scalar equations to systems entails, from a spectral-theoretic point of view,  a step change both in conceptual and technical difficulty. Systems exhibit properties that are totally different from those of scalar equations, and this is reflected in highly nontrivial challenges in the matters of propagator constructions and spectral asymptotics, see, \eg\cite[Section~11]{CDV}. In this paper we address some of these challenges for Lorentzian Dirac operators, through the prism of their Riemannian counterparts.  Our first main result is the construction of evolution operators for the massless Dirac equation --- a \emph{system} of PDEs --- on globally hyperbolic Lorentzian manifolds with compact Cauchy surface (precise definitions will be given in subsequent sections).  

\begin{theorem}\label{thm:Cauchy}
The Cauchy evolution operator for the reduced Dirac equation (see Subsection~\ref{sec:ident}) on a globally hyperbolic Lorentzian manifold $\M$ of dimension 4 and  with compact Cauchy hypersurface can be written as
\begin{equation*}
\label{main theorem 1 equation 1}
U(t;s)= U^{(+)}(t;s)+U^{(-)}(t;s)
\end{equation*}
modulo an integral operator with infinitely smooth kernel,
where
\begin{multline}
\label{main theorem 1 equation 2}
U^{(\pm)}(t;s)= \frac{1}{(2\pi)^{3}}\int_{\T^*\M} \hspace{-3mm} e^{i\varphi^\pm(t,x;s,y,\eta)}\mathfrak{a}^\pm(t;s,y,\eta)
\\
\times
\chi^\pm(t,x;s,y,\eta)\,w^\pm(t,x;s,y,\eta)\,(\cdot)\,\rho(y)dy\,d\eta.
\end{multline}
Here $\varphi^\pm$ are the positive and negative Levi-Civita phase functions given by Definition~\ref{definition levi civita phase functions}, $\chi^\pm$ are cut-off functions,  $w^\pm$ are related to $\varphi^\pm$ via \eqref{definition weight w pm},  and $\mathfrak{a}^\pm$ are invariantly defined matrix-functions determined via an explicit algorithm.
\end{theorem}

The adoption of distinguished complex-valued phase functions allows us to define the functions $\mathfrak{a}^\pm$ uniquely in an invariant fashion, and view them as the ``full symbols'' of the Cauchy evolution operator.  Note that, in general, there is no invariant notion of full (or subprincipal) symbol for a Fourier integral operator.

\subsection{Feynman propagator}

The decomposition into `positive' and `negative' propagators achieved in Theorem~\ref{thm:Cauchy},  resulting from the construction of suitably devised one-parameter families of `Lorentzian' pseudodifferential projections,  can be further exploited to construct Feynman parametrices explicitly, in the spirit of the above discussion.

\medskip

The Feynman propagator is a fundamental solution
of the Dirac operator propagating positive energy solutions to the future and negative energy solutions to the past, introduced by Richard Feynman in flat space to describe the physics of
the electron in Quantum Electrodynamics.   From a microlocal perspective, the Feynman propagator can be viewed, modulo smoothing, in terms of the so-called Feynman parametrices,  one type of \emph{distinguished parametrices} defined by Duistermaat and H\"ormander in~\cite[Section~6.6]{DuHo} for general pseudodifferential operators of real principal type.
\begin{definition}\label{def:Feynman para}
Let $\G_F$ be a parametrix for the reduced Dirac operator $\overline{\Dir}_\M$ (see Definition~\ref{def:reduced Dirac}), \textit{i.e.}, a continuous linear map $\G_F:C^\infty_c\big(\RR,\Gamma_c(\S\M|_{\Sigma_0})\big)\to C^\infty\big(\RR,\Gamma_c(\S\M|_{\Sigma_0})\big)$ satisfying
$$\overline{\Dir}_\M \G_F = \mathrm{Id} \mod \Psi^{-\infty}\,.$$
We say that $\G_F$ is a \textit{Feynman parametrix} (or {\em Feynman inverse}) if
\begin{multline*}
\operatorname{WF}(\G_F)
\\
= \Delta^*\cup \{(X,Y,k_X,k_Y)\in \T^*(\M\times\M)\backslash\{0\}|\ (X,k_X) \sim (Y,k_Y) \,,\; X \prec Y \}\,,
\end{multline*}
where: $\Delta^*$ is the diagonal in $\T^*(\M\times\M)\backslash\{0\}$; $(X, k_X ) \sim (Y, k_Y )$ means that $X$ and $Y$ are connected by a lightlike geodesic and $k_Y$ is
the parallel transport of $k_X$ from $X$ to $Y$ along said geodesic; $X\prec Y$ means that $X$ comes strictly before $Y$ with respect to the natural parameterisation of the geodesic.
\end{definition}

Feynman parametrices are important in many respects, within and beyond QFT.  Remarkably, they play a central role in the index theory of Lorentzian Dirac operators on globally hyperbolic manifolds with compact spacelike Cauchy boundary.  Indeed,  the Lorentzian Dirac operator is Fredholm if and only if it admits a Feynman inverse that is Fredholm and their indices are related as ${\rm ind}(\G_F)=-{\rm ind}(\overline{\Dir}_\M)$. Of course,  not any Feynman inverse can be expected to be Fredholm and appropriate boundary conditions have to be imposed, see \eg\cite{BaStr,BaStr2,ShWro}.  B\"ar and Strohmaier have shown \cite{BaStr} that the index of the Lorentzian Dirac operator can be expressed purely in terms of topological quantities associated with the underlying manifold --- very much like in the celebrated Atiyah–Patodi–Singer's theorem~\cite{APS}, its Riemannian counterpart. The resulting Lorentzian index theorem is a powerful analytical tool, which has, for example,  been applied to rigorously describe the gravitational chiral anomaly in QFT~\cite{BaStrCh}.  

\medskip

Our second main result is expressed by the following theorem.
\begin{theorem}\label{thm:Feynman}
The Feynman propagator can be represented, modulo an integral operator with infinitely smooth kernel,  as
\[
\G_F(t,s) = \theta(t-s)U^{(+)}(t;s)-\theta(s-t)U^{(-)}(t;s),
\]
where $U^{(\pm)}$ are defined in accordance with \eqref{main theorem 1 equation 2} and $\theta$ is the Heaviside theta function.
\end{theorem}

Theorem~\ref{thm:Feynman} takes earlier constructions of Feynman propagators on globally hyperbolic spacetimes via microlocal techniques --- see,  in particular, \cite{gerardDirac} and \cite{FeyStro}  --- further, by making them, in a sense, more explicit.  The adoption of distinguished geometric phase functions is particularly convenient in carrying out calculations, which result, in turn, in geometrically meaningful invariant quantities.

\subsection{Hadamard states}

As an application, we discuss the relation between our results and the construction of \emph{Hadamard states}.  

Within the quantisation scheme of algebraic quantum field theory in curved spacetimes, a key role is played by \emph{quantum states}, positive normalised linear functionals on the algebra of observables.  The issue one encounters is that the singular structure of an arbitrary state, albeit partially constrained by the equations ruling the dynamics of the quantum field, can in general be quite wild, and this quickly becomes a serious technical limitation in the manipulation of states for the purpose of describing the underlying physics.  For this reason, it was suggested to identify a physically meaningful subset of states --- the \emph{Hadamard states} --- by prescribing that their singularity structure mimic the ultraviolet behaviour of the Poincar\'e vacuum. This translates into a condition on the wavefront set of the 2-point distribution of the quantum state.  

We state here the Hadamard condition for the Dirac field, postponing until Section~\ref{sec:Hadam} precise definitions and a more complete discussion.  We refer the reader to~\cite{GerardBook, aqft2} for textbooks on the subject, to~\cite{CQF1,review} for recent reviews and~\cite{cap2,simogravity,DHP,DMP3, CSth,Junker:2001gx} for some applications. 
\begin{definition}\label{def: Hadamard state orig}
A quasifree state  $\omega$ on the algebra of Dirac field $\mathcal A$ is \textit{Hadamard} if its 2-point distribution $\omega_2$ defined in accordance with \eqref{n point function} satisfies
\begin{equation}
\label{hadamard condition omega 2}
\operatorname{WF}(\omega_2)=\{(X,Y,k_X,-k_Y)\in \T^*(\M\times\M)\backslash\{0\}|\ (X,k_X)\sim(Y,k_Y),\  k_X\rhd 0\},
\end{equation}
where $(X,k_X)\sim(Y,k_Y)$ means that $X$ and $Y$ are connected by a lightlike geodesic and $k_Y$ is the parallel transport of $k_X$ from $X$ to $Y$ along said geodesic, whereas $k_X\rhd 0$ means that the covector $ k_X$ is future pointing.
\end{definition}

Condition \eqref{hadamard condition omega 2} allows for the 
construction of Wick polynomials via a local and covariant scheme 
\cite{HW,IgorValter},  ensuring,  among other things,  the finiteness of the quantum fluctuations \cite{FV}.

Despite being structurally so important, Hadamard states are known to be rather elusive
to pin down,  especially when the spacetime is not static. 
Although their existence is guaranteed quite generally by \emph{non-constructive} deformation arguments~\cite{defarg,defargNorm,defargProc}, for practical purposes a constructive scheme is often needed.

In the last decade G\'erard and Wrochna exploited the pseudodifferential calculus to construct Hadamard states in many different contexts, see \eg\cite{gerard0,gerard1,gerard2,gerard4,gerard5} for the scalar wave equation, \cite{gerard3,wrochna1} for analytic Hadamard states, \cite{gerardHaw,gerardUnruh} for the Hawking and Unruh radiation, \cite{gerardDirac,gerardUnruh} for Dirac fields,  and \cite{gerardYM} for linearised Yang-Mills fields.  Their approach, which is essentially pseudodifferential in nature,  consists in constructing {\em Cauchy surface covariances} associated with a (quasifree) state. 

As a by-product,  our propagator construction provides us with a global, invariant representation of pseudodifferential operators in terms of oscillatory integrals with geometric phase functions, which we can use to explicitly construct Cauchy surface covariances of Hadamard states globally and invariantly.
Our result is summarised by the following theorem.  Once again, we refer to future sections for detailed definitions (see, in particular, Section~\ref{sec:appl}).

\begin{theorem}\label{thm:Hadamard}
Let $\tilde \lambda^\pm\in\Gamma'_c\big(\S\M|_\Sigma\otimes \S\M|_\Sigma \big)$ be defined as
\begin{align*}
\tilde\lambda^\pm(\Phi_1, \Phi_2)&:=\int_{\Sigma}
\fiber{{\rm tr}_{\Sigma_t}\Phi_1}{\gamma_\M(\partial_t)\Pi_\pm(t){\rm tr}_{\Sigma_t}\Phi_1} d\vol_{\Sigma_t}
\end{align*}
where $\Pi_\pm(t)$ are defined in accordance with Definition~\ref{definition left and right projections},  ${\rm tr}_{\Sigma_t}$ denotes the restriction to $\Sigma_t$, and $\gamma_M$ is the Clifford multiplication.
Then, 
\begin{equation}\label{eq:Cauchy surf cov}
\lambda^{\pm}(\Phi_1\oplus \Phi_1',\Phi_2\oplus \Phi_2'):=\tilde\lambda^\pm(\Phi_1, \Phi_2)
+\scalar{\Phi_1'}{\Phi_2'}
-\tilde\lambda^\pm(\Upsilon^{-1}\Phi_1', \Upsilon^{-1}\Phi_2')
\end{equation}
 define Cauchy surface covariances of quasifree Hadamard states for the reduced Dirac operator $\overline \Dir_\M$.  Here $\Upsilon$ is the adjunction map defined in Section~\ref{sec:Dirac} and $\scalar{\cdot}{\cdot}$ is the scalar product~\eqref{eq:Herm prod}.
\end{theorem}

\subsection{Structure of the paper}
The paper is structured as follows.  

Section~\ref{sec:preliminaries} contains a brief summary of the relevant notions of Lorentzian spin geometry needed throughout the paper. 

Section~\ref{sec:Geom constr prop} is the core of the paper.  In Subsection~\ref{sec:ident} we reduce the Cauchy problem for the Dirac operator to the Cauchy problem for a one-parameter family of Riemannian Dirac-type operator. In Subsection~\ref{sec:Cauchy ev op} we introduce the notion of positive/negative Cauchy evolution operators and in Subsection~\ref{sec:Lor pdo proj} we discuss the existence of appropriate `Lorentzian' (time-dependent) pseudodifferential projections which will implement the initial conditions of our evolution operators construction. Finally in Subsection~\ref{Construction of the propagators} we provide an algorithm to construct the integral kernel of the positive/negative Cauchy evolution operators. 

We conclude our paper with  Section~\ref{sec:appl}, devoted to the applications in quantum field theory on curved spacetimes. After recalling the notions of CAR algebra of Dirac solutions in Subsection~\ref{sec:Dir alg sol} and the basic properties of Hadamard states in Subsection~\ref{sec:Hadam},  Subsections~\ref{sec:proof1} and~\ref{sec:proof2} discuss the explicit invariant construction of Cauchy surface covariances and Feynman propagators for Dirac fields. 

\subsection*{General notation and conventions}
\begin{itemize}
\item[-] $(\M,g=-\beta^2dt^2+h_t)$ denotes a Cauchy compact, globally hyperbolic manifold of dimension $4$.
\item[-] $t:\M\to\RR$ denotes a  Cauchy temporal function (\cf Definition~\ref{def:Cauchy temporal}) and we denote Cauchy hypersurfaces by $\Sigma_s:=t^{-1}(s)$, for any $s\in\RR$. 
\item[-] Given a Cauchy temporal function $t$, we write $X=(t,x)$ for any $X\in \M$.
\item[-] The set $J^+(p)$ (\textit{resp.} $J^-(p)$) denotes the causal future (\textit{resp}. past) of $p\in \M$.

\item[-] $\T\Sigma$ is the tangent bundle of $\Sigma$ and we adopt the notation $\T'\Sigma:=\T\Sigma\setminus\{0\}$.
\item[-] Given a vector bundle $\E$ over $\M$, we denote
by $\Gamma(\E)$, $\Gamma_c(\E)$ and  $\Gamma_{sc}(\E)$ the linear spaces of smooth, \resp compactly supported, \resp spacially compact smooth 
sections of $\E$. 
\item[-] $\S\M$ denotes the spinor bundle over $(\M,g)$ (see Definition~\ref{def:spinor bundle}), $\gamma_\M$ is the Clifford multiplication,  $\Dir_\M : \Gamma(\S\M)\to\Gamma(\S\M)$ denotes the Lorentzian Dirac operator (see Definition~\ref{def:Dir}) and $\Dir_t$ the Riemannian Dirac operator on the spinor bundle $\S\Sigma_t$ over $(\Sigma,h_t)$.
\item[-] $\sharp:\Gamma(\T^*\M)\to\Gamma(\T\M)$ and its inverse  $\flat:\Gamma(\T\M)\to\Gamma(\T^*\M)$ denote the standard (fiberwise) musical isomorphisms associated with a given metric $g$ on $\M$.
\item[-]
We denote by $\Psi^k(\Sigma;\mathbb{C}^2)$ the space of polyhomogeneous pseudodifferential operators of order $k$ acting on 2-columns of complex-valued scalar functions on $\Sigma$.   Given $A\in\Psi^k(\Sigma;\mathbb{C}^2)$  we denote its principal symbol as $A_{\rm prin}$. Moreover, by $A=B \mod \Psi^{-\infty}$ we mean that $A-B$ is an integral operator with infinitely smooth kernel in all variables involved.
\item[-] For a vector-valued distribution $u\in\Gamma'_c(\E)$, we adopt the standard convention that the wavefront set $\operatorname{WF}(u)$ is the union of the wavefront sets of its components in an arbitrary but fixed local frame.
\item[-] $\overline\G$ denotes the causal propagator for the reduced Dirac operator $\overline{\Dir}_\M$ (see Definition~\ref{def:reduced Dirac}).
\item[-] We adopt Einstein's summation convention over repeated indices. We will normally denote spacetime tensor indices by roman letters and spatial tensor indices (related to quantities on $\Sigma$) by Greek letters.
\end{itemize} 

\subsection*{Acknowledgments}
We are grateful to Alberto Bonicelli and to Christian Gérard for helpful discussions related to the topic of this paper.

\subsection*{Funding}
MC was partially supported by the Leverhulme Trust Research Project Grant RPG-2019-240 and EPSRC grant EP/X01021X/1: both are gratefully acknowledged.
SM was supported by the DFG research grant MU 4559/1-1 ``Hadamard States in Linearized Quantum Gravity'' and is grateful for the support of the INdAM-GNFM funding ``Feynman propagator for Dirac fields: a microlocal analytic approach'' during the final stages of this project.

\section{Preliminaries in Lorentzian spin geometry}
\label{sec:preliminaries}

The aim of this section is to recall some basic results for Lorentzian Dirac operators on globally hyperbolic manifolds. 
We refer the reader to \cite{BaWave} and~\cite{Spin} for a more detailed introduction to Lorentzian geometry, spin geometry and Dirac operators on pseudo-Riemmanian manifolds.

\subsection{Globally hyperbolic manifolds}\label{Geometric setting}

In the following,  $(\M, g)$ will be a $4$-dimensional Lorentzian manifold of metric signature $(-,+	, +, +)$.  Within this class, we will focus on \emph{globally hyperbolic} Lorentzian manifolds: these are a distinguished class of Lorentzian manifolds on which initial value problems for hyperbolic equations of mathematical physics are well-posed.

\begin{definition}\label{def:GH}
	A \emph{globally hyperbolic manifold} is an oriented,
	time-oriented, Lorentzian smooth manifold $(\M,g)$ such that
	\begin{enumerate}[(i)]
		\item $\M$ is causal, \ie there are no closed causal curves;
		\item for every pair of points $p,q\in \M$,  the set $J^+(p)\cap J^-(q)$ is either compact or empty. 
	\end{enumerate}
\end{definition}

 In his seminal paper~\cite{Geroch},  Geroch established that the global hyperbolicity of a Lorentzian manifold is equivalent to the existence of a Cauchy hypersurface. 
\begin{definition}
A subset $\Sigma\subset\M$ of a spacetime $(\M,g)$ is called {\em Cauchy hypersurface} if it intersects exactly once any inextendible future-directed smooth timelike curve.
\end{definition}
We remind the reader that a smooth curve $\gamma : I \to \M$ with $I \subset \mathbb{R}$  open interval is said to be {\em inextendible}  if there is no  {\em continuous} curve $\gamma': J \to \M$ defined on  an open interval  $J \subset \mathbb{R}$ such that $I \subsetneq J$ and $\gamma'|_I =\gamma$.
\begin{theorem}[\protect{\cite[Theorem 11]{Geroch}}]\label{thm:Geroch} A Lorentzian manifold $(\M,g)$ is globally hyperbolic if and only if it contains a Cauchy hypersurface.
\end{theorem}

It turns out that Cauchy hypersurfaces of $(\M,g)$ are codimension-1 topological submani\-folds of $\M$  homeomorphic to each other. As a byproduct of Geroch’s theorem, it follows
that a globally hyperbolic manifold $(\M,g)$ admits a continuous foliation in Cauchy hypersurfaces
$\Sigma$, namely $\M$ is homeomorphic to $\RR\times \Sigma$.
This is established by finding a {\em Cauchy time function}, \ie a countinuous function $t:\M\rightarrow\mathbb{R}$ which is strictly increasing on any future-directed timelike curve and whose level sets $t^{-1}(t_0)$, $t_0\in \mathbb{R}$, are Cauchy hypersurfaces  homeomorphic to $\Sigma$.
 Geroch's splitting is topological in nature,  and the fact that it can be done in a smooth way has been part of the mathematical folklore whilst remaining an open problem for many years. Only recently Bernal and S\'anchez \cite{BeSa} ``smoothened'' the result of Geroch by introducing the notion of a \emph{Cauchy temporal function}.

\begin{definition} \label{def:Cauchy temporal}
Given a connected time-oriented smooth
Lorentzian manifold $(\M,g)$,  we say that a smooth function  $t : \M \to \RR$ is a {\em Cauchy temporal function}  if
the embedded codimension-1 submanifolds given by  its level sets are smooth Cauchy hypersurfaces and $dt^\sharp$ is past-directed and timelike. 
\end{definition} 

\begin{theorem}[\protect{\cite[Theorem 1.1, Theorem 1.2]{BeSa}}, \protect{\cite[Theorem 1.2]{BeSa2}}]\label{thm: Sanchez}
Any globally hyperbolic manifold $(\M,\overline g)$  admits a Cauchy temporal function. In particular, it is isometric to the (smooth) product manifold $\RR \times \Sigma$ with metric 
\begin{equation}  
 \label{GHmetric}
g= - \beta^2 d t^2 + h_t,
\end{equation}
where $t:\RR\times \Sigma \to \RR$ is a Cauchy temporal function acting as a projection onto the first factor,	 $\beta : \RR \times \Sigma \to \RR$ is a positive smooth function called {\em lapse function},  and $h_t$ is a one-parameter family of Riemannian metrics on $\Sigma$.\\
Moreover, if $\Sigma \subset \M$ is a spacelike Cauchy hypersurface, then there exists a Cauchy temporal function $t$ such that $\Sigma$ belongs to the foliation $t^{-1}(\RR)$.
\end{theorem}

The class of globally hyperbolic manifolds is non-empty and contains many important examples of spacetimes relevant to
general relativity and cosmology, \eg
the de Sitter spacetime,  a maximally symmetric solution of Einstein's equations with
positive cosmological constant. 
Of course, given any $n$-dimensional compact manifold $\Sigma$,  an open interval $I\subseteq\mathbb{R}$ and a smooth one-parameter family of Riemannian metrics $\{h_t\}_{t\in I}$, the $(n+1)$-dimensional Lorentzian manifold $(I\times\Sigma, g=-dt^2+h_t)$ is globally hyperbolic.

%

\subsection{The Lorentzian Dirac operator}\label{sec:Dirac}

%
As before, let $(\M,g)$ be a globally hyperbolic manifold of dimension $4$. Since any $3$-manifold is parallelisable,  then $\M$ is parallelisable as well on account Theorem~\ref{thm: Sanchez}. This guarantees the existence of a spin structure on $(\M,g)$, \ie a double covering map from the $\Spin_0(1,3)$-principal bundle $\P_{\rm 
Spin_0}$  to the bundle of positively oriented tangent frames $\P_{\rm SO^+}$ of 
$\M$ such that the following diagram is commutative:
\begin{flalign*}
\xymatrix{
\P_{\Spin_0} \times \Spin_0(1,3) \ar[d]_-{} \ar[rr]^-{} && \P_\Spin \ar[d]_-{} 
\ar[drr]^-{}   \\
\P_{\rm SO^+} \times \textnormal{SO}^+(1,3)   \ar[rr]^-{} &&  \P_{\rm SO^+}  
\ar[rr]^-{} &&  \M \,.
}
\end{flalign*}
By \emph{frame} at a point $p\in M$, we mean a positively oriented and positively time-oriented orthonormal (in the Lorentzian sense) collection vectors $e_j$,  $j=0,\ldots, 3$, in the tangent space $\T_p\M$.  By a \emph{framing} of $\M$ we mean a choice of frame $\{e_j(p)\}_{j=0}^3$ at every point $p\in \M$ depending smoothly on the base point $p$. For future convenience, we define
\begin{align}\label{Eq: Sigma unit normal}
e_0 := \beta^{-1}\partial_t\,
\end{align}
to be the global unit normal to the foliation $\Sigma_t := \{t\}\times \Sigma$.  Unless otherwise stated, in what follows a framing $\{e_j\}_{j=0}^n$ will always be such that the vector field $e_0$ is given by \eqref{Eq: Sigma unit normal}.  Of course,  formulae \eqref{GHmetric} and \eqref{Eq: Sigma unit normal} imply $g(e_0,e_0)=-1$.

\begin{definition}\label{def:spinor bundle}
We call {\em (complex) spinor bundle} $\S\M$ the complex vector bundle
$$\S\M:=\Spin_0(1,3)\times_\zeta \CC^4$$
where $\zeta: \Spin_0(1,3) \to \textnormal{Aut}(\CC^4)$ is the complex 
$\Spin_0(1,3)$ representation.
\end{definition}

Notice that, since $\M$ is parallelisable,  the spinor bundle $\S\M$ is trivial. 
As customary, we equip the spinor bundle with the following structures:
\begin{itemize}
\item[-] a natural $\Spin_0(1, 3)$-invariant indefinite inner product
\begin{equation*}\label{eq: spin prod}
\fiber{\cdot}{\cdot}_p: \S_p\M \times \S_p\M \to \CC;
\end{equation*}
\item[-] a \textit{Clifford multiplication}, \ie  a fibre-preserving map 
$$\gamma_\M\colon \T\M\to \text{End}(\S\M)$$ 
which satisfies, for every $p \in \M$, $u, v \in \T_p\M$ and $\psi,\phi\in \S_p\M$, the identities
$$  \gamma_\M(u)\gamma_\M(v) + \gamma_\M(v)\gamma_\M(u) = -2g(u, v)\Id_{\S_p\M} $$ and $$
\fiber{\gamma_\M(u)\psi}{\phi}_p=\fiber{\psi}{\gamma_\M(u)\phi}_p\,.$$
\end{itemize}
Using the spin product~\eqref{eq: spin prod},  one defines the \textit{adjunction map} to be the complex anti-linear vector bundle isomorphism
\begin{equation*}\label{eq:adj map}
\Upsilon_p:\S_p\M_g\to \S^*_p\M_g  \qquad \psi \mapsto \fiber{\psi}{\cdot}\,,
\end{equation*}
where  $\S^*_p\M_g$ is the so-called \textit{cospinor bundle}, \ie  the dual bundle to $\S_p\M_g$.

\begin{definition} \label{def:Dir}
The \textit{(Lorentzian) Dirac operator} $\Dir_\M$ is defined as the 
composition of the metric connection $\nabla^{\S\M}$ on $\S\M$, obtained as lift 
of the Levi-Civita connection on $\T\M$, and the Clifford multiplication:
$$\Dir_\M:=\gamma_\M\circ\nabla^{\S\M} \colon \Gamma(\S\M) \to \Gamma(\S\M)\,.$$
\end{definition}

 Given a framing $\{e_j\}_{j=0}^3$ of $\M$,  the 
Dirac operator reads
\begin{align*}
\Dir_\M  = \sum_{j=0}^{3}  c_j\, \gamma_\M(e_j) \nabla^{\S\M}_{e_j}\, 
\end{align*}
where
$c_j:=g(e_j,e_j)=1-2\delta_{0j}$.


\begin{theorem}[\protect{\cite[Theorem 5.6 and Proposition 5.7]{BaGreen}}]
The Cauchy problem for the Dirac operator $\Dir_\M$ is well-posed, \ie for every $s\in\RR$ and every pair $(f,h)\in\Gamma_c(\S\M)\times\Gamma_{c}(\S\M_{|_{\Sigma_s}})$  there exists a unique $\psi\in\Gamma_{sc}(\S\M)$ satisfying the Cauchy problem
\begin{equation}\label{Cauchysmooth}
\begin{cases}{}
{\Dir_\M}\psi=f  \\
\psi_{|_{\Sigma_s}} = h
\end{cases} 
\end{equation}
and  the solution map $(f,h)\mapsto\psi$ is continuous in the standard topology of smooth sections. 
 \end{theorem}

It is well known that,  as a by-product of the well-posedness of the Cauchy problem, one obtains the existence of the advanced and the retarded Green operators. In particular, their differences, called the {\it causal propagator}, characterises the solution space of $\Dir_\M$.

\begin{proposition}[\protect{\cite[Theorem 5.9 and Theorem 3.8]{BaGreen}}]\label{prop:Green}
\noindent
\begin{enumerate}[(a)]
\item
The Dirac operator $\Dir_\M$ is Green-hyperbolic,  \ie there exist linear maps $\G^\pm\colon \Gamma_{c}(\S\M)  \to \Gamma_{sc}(\S\M)$, called 
\textit{advanced ($-$) and retarded (+) Green operators},
 satisfying
\begin{enumerate}[(i)]
\item
$(\Dir_\M\circ \G^\pm) f=f$ and $(\G^\pm\circ\Dir_\M)f=f$ for all $ f \in 
\Gamma_{c}(\S\M)$;
\item 
$\supp(\G^\pm f ) \subset J^\pm (\supp f )$ 
for all $f\in\Gamma_{c}(\S\M)$.
\end{enumerate}
\item Let $\mathsf{G}:=\mathsf{G}^+-\mathsf{G}^-$ be the causal propagator. Then the sequence
\begin{align*}
	0\to
	\Gamma_{c}(\S\M)
	\stackrel{\Dir_\M}{\to}\Gamma_{c}(\S\M)
	\stackrel{\mathsf{G}}{\to}\Gamma_{sc}(\S\M)
	\stackrel{\Dir_\M}{\to}\Gamma_{sc}(\S\M)\to 0
\end{align*}
is exact. In particular, 
\begin{align*}
	\sol_{sc}(\Dir):=\Gamma_{sc}(\S\M)\cap\ker\Dir_M
	\simeq{\rm ran\,}\G\simeq \rquot{\Gamma_{c}(\S\M)}{\Dir_\M\Gamma_{c}(\S\M)}\,.
\end{align*}
\end{enumerate}
\end{proposition}

\section{Geometric construction of global propagators}
\label{sec:Geom constr prop}
This section is the core of the paper. As a preliminary step, we will first reduce the Cauchy problem for the Dirac operator to the Cauchy problem for a one-parameter (time-dependent) family of Riemannian Dirac-type operators. Then, we will introduce the notion of positive/negative Cauchy evolution operators and we discuss the existence of `Lorentzian' (time-dependent) pseudodifferential projections that will implement the initial conditions of our evolution operators construction. At that point, we will have at our disposal all the ingredients to formulate an explicit geometric algorithm to construct the integral kernel of the positive/negative Cauchy evolution operators. 

\subsection{Geometric identification}\label{sec:ident}

The aim of this section is to reduce the Cauchy problem for the Lorentzian Dirac operator to the Cauchy problem for a one-parameter family of Riemannian Dirac operators defined on the reference submanifold $\Sigma$ of $\M$. To this end, we will mimic the arguments from~\cite{DiracAPS} (see also~\cite{Bar1,MollerMIT,IBVPg,DiracMIT}).
 Let us denote by
\begin{itemize}
\item $\nabla^{\S\Sigma_t}$ the spin connection on  $\S\Sigma_t$, obtained as lift of the Levi-Civita connection on $\T\Sigma_t$ associated with the Riemannian metric $h_t$ (recall \eqref{GHmetric});
\item $\gamma_{\Sigma_t}$ the Clifford multiplication on $\Sigma_t$;
\item $t \mapsto \Dir_t$  the one-parameter family of Riemannian Dirac operators on $(\Sigma, h_t)$ defined by 
$$\Dir_t\define \gamma_{\Sigma_t}\circ \nabla^{\S\Sigma_t}\,.$$
\end{itemize}
As shown in~\cite[Section 3]{Bar1}, the spin connection $\nabla^{\S\M}$ on $\S\M$ and the spin connection $\nabla^{\S\Sigma_t}$ on $\S\Sigma_t$ are related as
\begin{align*}
	\nabla^{\S\M}_{e_j}\psi
	=\nabla^{\S\Sigma_t}_{e_j}\psi
	-\frac{1}{2}\gamma_\M(e_0)\gamma_\M(\nabla^\M_{e_j}e_0)\psi\,,
\end{align*}
for any $\psi\in\Gamma(\S\M)$ and any ${e_j}\in\Gamma(\T\M)$ tangent to $\Sigma_t$. 
It then follows that
\begin{align}\label{Eq: DM,Dbullet relation}
	\gamma_\M(e_0) \Dir_\M= 
	 \begin{pmatrix} - \nabla^{\S\M}_{e_0} - i \Dir_t +\frac{3}{2}H_t &0\\0& - \nabla^{\S\M}_{e_0} +i \Dir_t +\frac{3}{2}H_t\end{pmatrix},
\end{align}
where $H_t\define -\frac{1}{3}\sum_{j=1}^3h_t(e_j,\nabla^{M}_{e_j}e_0)$ is the mean curvature of $\Sigma_t$. 
Using the above identity, the reduction of the Cauchy problem is achieved in a two step procedure: 

\medskip

{\bf Step 1}.  First, we perform a conformal transformation of the metric 
\[
 g \mapsto \hat{g}:= \beta^{-2}g=-dt^2 + \beta^{-2}h_t =: -dt^2 + \hat{h}_t \,.
\]
As a result, one immediately obtains the following.
 \begin{lemma}\label{lem:1}
 The Cauchy problem for Dirac operator $\Dir_\M$ on $(\M,g)$ is equivalent to the Cauchy problem for the Dirac operator $\hat\Dir_\M$ on $(\M,\hat{g})$, namely there is a one-to-one correspondence $\hat{\psi}=\beta \psi$ between solutions of the Cauchy problems
 \begin{align*}
		\begin{cases}
		 \Dir_\M \psi=0\\
		\psi|_{\Sigma_0}=\psi_0 
		\end{cases}
	 \qquad \Longleftrightarrow \qquad
		\begin{cases}
		\hat \Dir_\M \hat\psi=0\\
		\hat\psi|_{\Sigma_0}=\left. \beta^{\frac{3}{2}}\right|_{t=0}\psi_0 
		\end{cases},
	\end{align*}	
$\psi_0\in  \Gamma_{\mathrm{c}}(\S\M|_{\Sigma_0})$.
 \end{lemma}
 \begin{proof}
It is easy to see that the Clifford multiplication $\hat{\gamma}_\M$ and the spin connection $\hat{\nabla}^{\S\M}$ associated with $(\M,\hat g)$ are related to those associated with $(\M,g)$ as
\begin{align}
\label{final edits eq 1}
	\hat{\gamma}_\M(v)=\beta\gamma_\M(v)\,,\qquad
	\hat{\nabla}^{\S\M}_v=\nabla^{\S\M}_v
	+\frac{\beta}{2}\gamma_\M(X)\gamma_\M(\nabla^\M \beta^{-1})
	-\frac{\beta}{2}v(\beta^{-1})\,,
\end{align}
whereas the inner product $\fiber{\cdot}{\cdot}$ is invariant under conformal transformations.
The Dirac operators $\hat \Dir_\M$ and $\Dir_\M$ are then related as
\begin{align}\label{Eq: conformal Dirac operator}
	\hat{\Dir}_\M=\beta^{\frac{5}{2}}\,\Dir_\M\, \beta^{-\frac{3}{2}}\,.
\end{align}
The latter implies that if $\Dir_\M\psi=0$,  then $\hat{\Dir}_\M(\beta^{\frac{3}{2}}\,\psi)=0$,  which concludes the proof.
 \end{proof}
 
\begin{remark}
Note that $\hat{e}_j:=\beta \,e_j$, $j=0,1,2,3$,  is a framing on $(\M, \hat g)$.  Formulae \eqref{Eq: DM,Dbullet relation},  \eqref{Eq: conformal Dirac operator} and \eqref{final edits eq 1} imply 
\[
\hat{\Dir}_t=\beta^{\frac{5}{2}}\,\Dir_t\, \beta^{-\frac{3}{2}}.
\]
Moreover,  the vector field $\hat e_0=\partial_t$ is geodesic on $(\M,\hat{g})$,  namely, $\hat\nabla_{\hat e_0}\hat e_0=0$.
\end{remark} 
 
\medskip

{\bf Step 2}.  Next,  we consider the Dirac operator $\hat\Dir_\M$ \eqref{Eq: conformal Dirac operator} on the globally hyperbolic manifold $(\M,\hat{g})$. 
Recalling that the Dirac operator $\hat\Dir_\M$ on $(\M, \hat g)$ decomposes as~\eqref{Eq: DM,Dbullet relation},
let us turn $\hat\Dir_\M$ into an operator on $C^\infty(\RR,\S\M|_{\Sigma_0})$. To this end, we
 identify the spinor bundles $\S \Sigma_t$ on $(\Sigma_t,\hat{h}_t$) associated with $\Sigma_t$, $t\in \RR$,  by means of parallel transport $\wp_{0,t}\colon \S\M|_{\Sigma_t}\to \S\M|_{\Sigma_0}$ along the integral curves of the vector field $\partial_t$.
 The map $\wp_{0,t}$ is a linear isometry between spinor bundles which preserves the positive definite form $\fiber{\cdot}{\gamma_\M(\partial_t)\,\cdot}$ because $\partial_t$ is geodesic.  In order to lift the map $\wp_{0,t}$  to an isometry of Hilbert spaces, let $\rho\colon \M \to \mathbb R$ be the nonnegative smooth function defined by 
\begin{equation}
\label{definition funny rho}
 d\vol_{\Sigma_t} = \rho^2 \,d\vol_{\Sigma_0}\,. 
\end{equation} 
Then
\begin{align}\label{Eq: identification isomorphism}
	V_t\colon\Gamma(\S\M|_{\Sigma_t})\to \Gamma(\S\M|_{\Sigma_0})\,,\qquad
	V_t\psi\define \rho\, \wp_{0,t}\psi\,,
\end{align}
defines an isometry between $L^2(\S\M|_{\Sigma_t})$ and $L^2(\S\M|_{\Sigma_0})$. Indeed, 
\begin{align*}
	\scalar{V_t\psi_1}{V_t\psi_2}_{L^2(\S\M|_{\Sigma_0})}
	&=\int_\Sigma \fiber{\wp_{0,t}\psi_1}{\gamma_\M(\partial_t)\wp_{0,t}\psi_2}\rho^2 d\vol_{\Sigma_0}
	\\&=\int_{\Sigma_t}\fiber{\psi_1}{\gamma_\M(\partial_t)\psi_2}d\vol_{\Sigma_t}
	=\scalar{\psi_1}{\psi_2}_{L^2(\S\M|_{\Sigma_t})}\,.
\end{align*}

\begin{remark}
Of course, for ultrastatic globally hyperbolic manifold $(\M,\hat{g}=-dt^2+\hat{h})$ the linear isometry $V_t$ reduces to the identity map for all $t$. 
\end{remark}

One has the following.
\begin{lemma}\label{lem:redDir}
Let  $\hat{\Dir}_t$ be the Riemannian Dirac operator on $(\Sigma_t,\hat{h}_t)$ and let $V_t:L^2(\S\M|_{\Sigma_t}) \to L^2(\S\M|_{\Sigma_0})$ be the isometry defined by~\eqref{Eq: identification isomorphism}. Then, for any $(x,\eta)\in\T^*\Sigma$, the principal symbol of $\overline{\Dir}_t := V_t \hat\Dir_t V_t^{-1}$ satisfies
 $$ (\overline{\Dir}_t)_\mathrm{prin}(x,\eta)= (\hat{\Dir}_0)_\mathrm{prin} (x,\kappa_{0,t} \eta)\,,$$
 where 
 $\kappa_{0,t}:\T^*\M|_{\Sigma_t}\to\T^*\M|_{\Sigma_0}$ is the parallel transport along the integral curves of $\partial_t$ on $(\M,\hat{g})$.
\end{lemma}
The lemma above motivates the following definition.
\begin{definition}\label{def:reduced Dirac}
We call {\em reduced Dirac operator $\overline{\Dir}_\M$ for the metric $g$}  the operator defined by
$$\overline{\Dir}_\M:= \begin{pmatrix}- i\partial_t + \overline{\Dir}_t &0\\0& - i\partial_t - \overline{\Dir}_t\end{pmatrix} :C^\infty\big(\RR,\Gamma(\S\M|_{\Sigma_0})\big)\to C^\infty\big(\RR,\Gamma(\S\M|_{\Sigma_0})\big) \,. $$
\end{definition}
\begin{proof}[Proof of Lemma~\ref{lem:redDir}]
Since the map $\wp_{0,t}$ preserves the Clifford multiplication 
 \ie for any 
$v\in\Gamma(\T\M)$ and $\psi_0\in\Gamma(\S\M|_{\Sigma_t})$ we have
 $$\gamma_{\Sigma_0}(\kappa_{0,t} v) (\wp_{0,t}  \psi) =\wp_{0,t} \,\big(\gamma_{\Sigma_t}(v)\psi\big)\,,$$
 (see~\cite[Lemma 3.7]{defarg} for a proof),
then for every $(x,\eta)\in\T^*\Sigma_t$  we have 
	\begin{align*}
(\overline{\Dir}_t)_\mathrm{prin}(x,\eta) &= \wp_{0,t} (\hat{\Dir}_t)_\mathrm{prin}(x,\eta) \wp_{t,0} =  \wp_{0,t} \gamma_{\Sigma_t}(\eta^{\sharp_t}) \wp_{t,0} \\
&= \gamma_{\Sigma_0}(\kappa_{0,t}\eta^{\sharp_t})  = (\hat{\Dir}_0)_\mathrm{prin} (x,(\kappa_{0,t} \eta^{\sharp_t})^{\flat_0})  = (\hat{\Dir}_0)_\mathrm{prin} (x,\kappa_{0,t}\eta)\,,
	\end{align*}
	where here $\sharp_t$ is the musical isomorphism associated with $\hat h_t$. In the last equality we used that, for all $\eta\in \T_x^*\M$ and $v\in \T_x\M$, 
\begin{align*}
(\kappa_{0,t}\eta^{\sharp_t})^{\flat_0}(v)|_x
&=h_0(\kappa_{0,t}\eta^{\sharp_t},v)|_x
=\hat g_\M(\kappa_{0,t}\eta^{\sharp_t},v)|_{(0,x)}
=\hat g_\M(\eta^{\sharp_t},\kappa_{t,0}v)|_{(t,x)}
\\&=h_t(\eta^{\sharp_t},\kappa_{t,0}v)|_x
=\eta(\kappa_{t,0}v)|_x
=[\kappa_{0,t}\eta](v)|_x\,,
\end{align*}
where, with slight abuse of notation,  $\kappa_{0,t}\eta$ denotes the parallel transport of the covector $\eta$ along the unique integral curve of $\partial_t$ connecting $(t,x)$ and $(0,x)$. 
\end{proof}

We are finally in a position to state and prove the most important result of this subsection.
\begin{theorem}\label{thm:identif}
Let $(\M,g=-\beta^2dt^2+h_t)$ be a globally hyperbolic manifold.
The Cauchy problem 
	for the Lorentzian Dirac operator $\Dir_\M$ on $(\M,g)$
	is equivalent to the Cauchy problem  for reduced Dirac operator $\overline{\Dir}_\M$ for the metric $g$, namely,  
there exists a one-to-one correspondence $\overline{\psi}:=\beta V_t\psi \in C^\infty\big(\RR,\Gamma(\S\M|_{\Sigma_0})\big)$ between 	
	solutions $\psi\in\Gamma(\S\M)$ of 
	$$\begin{cases}
	\Dir_\M\psi=0\\
	\psi|_{\Sigma_0}=\psi_0
	\end{cases}
	$$
	and solutions $\overline\psi \in C^\infty\big(\RR,\Gamma(\S\M|_{\Sigma_0})\big)$ of
	$$\begin{cases}
	\overline\Dir_\M \overline\psi=0\\
	\overline\psi|_{\Sigma_0}=\beta\psi_0
	\end{cases}
	$$
for 	$\psi_0\in \Gamma(\S\M|_{\Sigma_0})$.
\end{theorem}
\begin{proof}
Since $\M$ is an even dimensional manifold, then $\S\M|_{\Sigma_t}=\S\Sigma_t\oplus\S\Sigma_t$ (see \eg~\cite{Bar1}), which implies that $\Gamma(\S\M)\simeq C^\infty\big(\RR,\Gamma(\S\M|_{\Sigma_t})\big)\simeq  C^\infty\big(\RR,\Gamma(\S\Sigma_t)\oplus \Gamma(\S\Sigma_t)\big)$.
A direct calculation leads to
\begin{align*}
	\partial_t(V_t\psi)
	=V_t\rho^{-1}\hat\nabla^{\S\M}_{\partial_t}(\rho\psi)\,,
\end{align*}
(recall~\eqref{definition funny rho} and~\eqref{Eq: identification isomorphism})
which,  in turn,  gives us
\begin{align*}
	V_t\hat\nabla^{\S\M}_{\partial_t}V_t^{-1}
	=\rho\,\partial_t\circ \rho^{-1}
	=\partial_t-\rho^{-1}[\partial_t,\rho]\,.
\end{align*}
At the same time, $\hat\nabla^{\S\M}_{\partial_t}\partial_t=0$ implies
\begin{align*}
	-nH_t=\operatorname{div}_M(\partial_t)
	=|h_t|^{-\frac{1}{2}}\partial_t|h_t|^{\frac{1}{2}}
	=2\rho^{-1}(\partial_t\rho)
	=2\rho^{-1}[\partial_t,\rho]\,,
\end{align*}
where $|h_t|\define \det h_t$.
Overall, combining~\eqref{Eq: DM,Dbullet relation} with the above identities, we get
\begin{align*}
	V_t\hat\Dir_\M V_t^{-1}
	=-\hat\gamma_\M(e_0)\begin{pmatrix}\partial_t +i \overline{\Dir}_t &0\\0& \partial_t -i \overline{\Dir}_t\end{pmatrix}\,.
\end{align*}
Hence $\hat\Dir_\M\psi=0$ is equivalent to
\begin{align*}
	\begin{pmatrix}\partial_t +i \overline{\Dir}_t &0\\0& \partial_t -i \overline{\Dir}_t\end{pmatrix}\overline{\psi}=0\,,
\end{align*}
with $\overline{\psi}\define V_t\psi\in C^\infty(\mathbb{R},\Gamma(\S\M|_\Sigma))$ which, on account of Lemma~\ref{lem:1},  concludes our proof.
\end{proof}

\subsection{The Cauchy evolution propagator}\label{sec:Cauchy ev op}
We are finally in a position to construct the {\em Cauchy evolution propagator} for the Dirac equation on our globally hyperbolic manifold $(\M,g)$.

\begin{assumption}
Further on,  we assume that $(\M,g)$ is Cauchy-compact, \ie $\M$ is foliated by closed Cauchy hypersurfaces.
\end{assumption}
\begin{remarks}
\noindent
\begin{enumerate}
\item
The compactness of $\Sigma$ ensures that $\overline{\Dir}_t$ has discrete spectrum accumulating to $\pm \infty$ and it automatically guarantees uniform estimates for the smoothing remainder in our propagator construction from Section~\ref{sec:Cauchy ev op}. Whilst this is not strictly necessary, and one may obtain similar results by imposing appropriate conditions on the metric (\eg, bounded geometry) or on the decay properties at infinity of elements of our function spaces, the amount of additional technical details needed for such generalisation is not balanced by a corresponding gain in terms of insight.  In the overall interest of clarity and readability, we refrain from removing this assumption in the current paper, with the plan to address more general settings elsewhere.

\item 
Let us point out that if $(\M,g)$ is a globally hyperbolic manifold with noncompact Cauchy hypersufaces $\Sigma$ then for any finite time interval $(t_1, t_2)\subset\RR$ there exists a globally hyperbolic manifold $\tilde{\M}$ with compact Cauchy surface $\tilde\Sigma$ such that the Cauchy problem~\eqref{Cauchysmooth} can be solved equivalently in $\tilde\M$. This can be seen as follows (see also~\cite{BaGreen}).
For sake of simplicity,  let us consider the homogeneous Cauchy data $(0,h)$ and let us define $\Sigma'$ to be the projection onto $\Sigma$ of the compact subset 
$J(\supp(h))\cap([t_1,t_2]\times\Sigma)$.
Then there exists a relatively compact open neighborhood $U$ of 
$\Sigma'$ in $\Sigma$ with smooth boundary $\partial U$.
Denoting the doubling of $U$ by $\tilde{\Sigma}$, we define $\tilde{\M}:=[t_1,t_2]\times\tilde{\Sigma}$.  By finite speed of propagation (see \eg~\cite[Proposition~3.4]{DiracMIT})  the support of $\psi$ is contained in $[t_1,t_2]\times \Sigma'$. Furthermore, since the solution is uniquely determined by the Cauchy data $(0,h)$,  any solution of the Cauchy problem in $(t_1,t_2)\times\tilde{\Sigma}$ 
is also a solution in $(t_1,t_2)\times \Sigma$.

The above argument tells us that if we are only interested in solving the Cauchy problem for the Dirac equation on a globally hyperbolic spacetime with noncompact Cauchy surface for finite times with initial data supported in an arbitrary but fixed relatively compact set, then it is enough to be able to solve the Cauchy problem on a Cauchy-compact spacetime. 
\end{enumerate}
\end{remarks}
In light of Section~\ref{sec:ident}, and of Theorem~\ref{thm:identif} in particular,  the study of the Cauchy problem for the Lorentzian Dirac operator $\Dir_\M$ on $(\M,g)$ reduces to the initial value problem for the hyperbolic systems
\begin{subequations}\label{eq:dirac reduced}
\begin{empheq}[left=\empheqlbrace]{align}
&- i \partial_t \phi_L + \overline \Dir_t \phi_L =0,\\
 & -i \partial_t \phi_R - \overline \Dir_t\phi_R =0,
\end{empheq}
\end{subequations}
subject to the initial conditions $\phi_L|_{\Sigma_0}=f_L\in C^\infty(\RR,\CC^2)$ and $\phi_R|_{\Sigma_0}=f_R\in C^\infty(\RR,\CC^2)$.
Since $\M$ is a 4-dimensional globally hyperbolic manifolds,  the bundle $\S\M$ is trivial and  $\overline \Dir_t$ is a $2\times 2$ self-adjoint matrix differential operator acting on sections of the trivial $\mathbb{C}^2$-bundle (2-columns of complex-valued scalar functions). 

\begin{notation}
\label{notation principal sumbol et al}
Before progressing, let us introduce some more notation.
\begin{enumerate}
\item 
We denote by $(\overline \Dir_s)_\mathrm{prin}$ the principal symbol of the operator $\overline \Dir_s$, viewed as an elliptic operator whose coefficients smoothly depend on $s$,  acting in $\Sigma$ (\cf Lemma~\ref{lem:redDir}).  A straightforward calculation shows that the eigenvalues of 
$(\overline \Dir_s)_\mathrm{prin}$
are
$$h^{(\pm)}(y,\eta;s):=\pm \sqrt{\hat h_s^{\alpha\beta}(y) \eta_\alpha\eta_\beta}\,.$$
We denote by $P^{(\pm)}(y,\eta;s)$ be the corresponding eigenprojections.
Of course,
\begin{equation*}\label{sum of  eigenprojections}
P^{(\pm)}(y,\eta;s)+P^{(\mp)}(y,\eta;s)=I
\end{equation*}
for every $s\in \RR$ and $(y,\eta)\in \T_y\Sigma_s\setminus \{0\}$,where $I$ is the $2\times 2$ identity matrix.

\item Further on, when we write$\mod \Psi^{-\infty}$ we mean that the pseudodifferential (resp.~Fourier integral) operator on the LHS of the equation that precedes it differs from the pseudo\-differential (resp.~Fourier integral) operator on the RHS by an integral operator with infinitely smooth kernel in all variables involved.
\end{enumerate}
\end{notation}

Reducing the original problem to the pair of first order systems \eqref{eq:dirac reduced} is particularly convenient, in that $(\overline \Dir_t)_\mathrm{prin}$ has \emph{simple} eigenvalues $h^{(\pm)}(t;x,\eta)=\pm \|\eta\|_{h_t}$.  This property is crucial in the propagator construction below.

\medskip

The task at hand is to construct,  in a global, explicit and invariant fashion,  an integral operator that maps initial data on a given Cauchy surface $\Sigma_s$ to full propagating solutions of equation~\eqref{eq:dirac reduced}.
More precisely,  given $s\in \mathbb{R}$, we will construct the (distributional) solution 
$
U_L(t;s)
$ and $
U_R(t;s)
$
of the operator-valued Cauchy problems
\begin{subequations}\label{eq:Cauchy probl U}
\begin{empheq}[left=\empheqlbrace]{align}
\label{propagator full}
\left(-i\partial_t+\overline  \Dir_t\right)U_L(t;s)=0 \mod \Psi^{-\infty}\\
\label{initial condition propagator full}
U_L(s;s)=\operatorname{Id}  \mod \Psi^{-\infty},
\end{empheq}
\end{subequations}
\begin{subequations}\label{eq:Cauchy probl U_R}
\begin{empheq}[left=\empheqlbrace]{align}
\label{propagator full Right}
\left(-i\partial_t-\overline  \Dir_t\right)U_R(t;s)=0 \mod \Psi^{-\infty}\\
\label{initial condition propagator full Right}
U_R(s;s)=\operatorname{Id}  \mod \Psi^{-\infty},
\end{empheq}
\end{subequations}
where $\operatorname{Id}$ is the identity operator in $L^2(\S\Sigma_s)\simeq L^2(\Sigma_s)\oplus L^2(\Sigma_s)$.
\begin{definition}
We call {\em Cauchy evolution propagator}  the Fourier integral operator $$U(t;s)=\begin{pmatrix}
 U_L(t;s)& 0 \\ 0 & U_R(t;s)
\end{pmatrix}$$ where $U_L(t;s)$ and $U_R(t;s)$ solve the operator-valued Cauchy problems~\eqref{eq:Cauchy probl U} and~\eqref{eq:Cauchy probl U_R}, respectively. 
\end{definition}

Without loss of generality, we shall focus on constructing $U_L$; $U_R$ can be obtained analogously. 
We will seek $U_L(t;s)$ in the form
\[
U_L(t;s)=U_L^{(+)}(t;s)+U_L^{(-)}(t;s) \mod \Psi^{-\infty},
\]
where
 $U_L^{(+)}(t;s)$ and $U_L^{(-)}(t;s)$ are required to satisfy the following conditions: 
\begin{enumerate}[(i)]
\item 
 \eqref{propagator full} individually,
 
 \item 
  the `joint' initial condition
\begin{equation}
\label{sum UL plus and minus gives identity}
U_L^{(+)}(s;s)+U_L^{(-)}(s;s)=\mathrm{Id} \mod \Psi^{-\infty}\,,
\end{equation}
\item
the `compatibility evolution conditions'
\begin{equation}
\label{initial condition equation 1}
U_L^{(\pm)}(t;t')U_L^{(\pm)}(t';s)=U_L^{(\pm)}(t;s) \mod \Psi^{-\infty}
\end{equation}
\begin{equation}
\label{initial condition equation 2}
U_L^{(\pm)}(t;t')U_L^{(\mp)}(t';s)=0 \mod \Psi^{-\infty},
\end{equation}
\end{enumerate} 
for all $t,t',s \in \RR$.

\begin{definition}
In order to conform to the Riemannian terminology,  we call  $U_L^{(+)}(t;s)$ (\resp $U_L^{(-)}(t;s)$) {\em the positive}  (\resp  {\em the negative}) {\em Dirac propagator} for the reduced Dirac operator $\overline{\Dir}_\M$. 
\end{definition}

Now, from \eqref{initial condition equation 1} we get 
\begin{equation*}
\label{initial condition equation 3}
U_L^{(\pm)}(t;t)U_L^{(\pm)}(t;s)=U_L^{(\pm)}(t;s)U_L^{(\pm)}(s;s) \mod \Psi^{-\infty},
\end{equation*} 
which, combined with~\eqref{propagator full},  implies
\begin{equation}
\label{initial condition equation 4}
-i\partial_t U_L^{(\pm)}(t;t)+[\overline\Dir_t,  U_L^{(\pm)}(t;t)]=0\mod \Psi^{-\infty}\,.
\end{equation}
Furthermore, \eqref{initial condition equation 1} and \eqref{initial condition equation 2} imply,  respectively,
\begin{align}
\label{initial condition equation 5}
& (U_L^{(\pm)}(t;t))^2=U_L^{(\pm)}(t;t) \mod \Psi^{-\infty},\\
\label{initial condition equation 6}
& U_L^{(\pm)}(t;t)U_L^{(\mp)}(t;t)=0 \mod \Psi^{-\infty}\,.
\end{align}

That one-parameter families $\{U^{(\pm)}_L(t;t)\}_{t\in \mathbb{R}}$ of pseudodifferential operators satisfying \eqref{sum UL plus and minus gives identity},  \eqref{initial condition equation 4}--\eqref{initial condition equation 6} 
exist is not at all obvious,  in that the above conditions yield a heavily overdetermined system of equations for the (matrix) symbol of $U^{(\pm)}_L(t;t)$.
In the next subsection we will show that such operators do, indeed,  exist, and construct their full symbols explicitly. 

\subsection{Lorentzian pseudodifferential projections}\label{sec:Lor pdo proj}

In this section we will construct time-dependent pseudodifferential projections which will serve as initial conditions for our propagator construction, in light of \eqref{sum UL plus and minus gives identity},  \eqref{initial condition equation 4}--\eqref{initial condition equation 6}. 

\subsubsection{Existence and characterisation}

We denote by $\Psi^k(\Sigma;\mathbb{C}^2)$ the space of polyhomogeneous pseudodifferential operators of order $k$ acting on 2-columns of complex valued scalar functions on $\Sigma$.  We also introduce refined notation for the principal symbol.  Namely, we denote by $(\cdot)_{\mathrm{prin},k}$ the principal symbol of the expression within brackets, regarded as an operator in $\Psi^{k}$.  To appreciate the need for this, consider two operators $A$ and $B$ of order $s$ with the same principal symbol. Then $A-B$ is, effectively, an operator of order $s-1$.  Hence,  $(A-B)_\mathrm{prin}=(A-B)_\mathrm{prin,s}=0$ but $(A-B)_\mathrm{prin,s-1}$ may not vanish. This refined notation will be used whenever there is risk of confusion.

The following is the main result of this subsection. Later in the paper, it will be specialised to the case $A_t=\pm \overline \Dir_t$.

\begin{theorem}
\label{theorem pseudodifferential projections}
Let $A_t\in C^\infty(\Sigma; \Psi^{1}(\Sigma,\mathbb{C}^2))$ be a one-parameter family of pseudodifferential operators of order $1$ acting on the trivial $\mathbb{C}^2$-bundle over $\Sigma$. Suppose $(A_t)_\mathrm{prin}$ has simple eigenvalues $h^{(\pm)}_A(x,\xi;t)$ and let $P^{(\pm)}_A(x,\xi;t)$ be the corresponding eigenprojections.

(a) 
There exist one-parameter families of pseudodifferential projections $$P_\pm(t)\in C^\infty(\mathbb{R};\Psi^0(\Sigma; \mathbb{C}^2))$$ satisfying
the following conditions:
\begin{equation}
\label{pseudodifferential projections equation 1}
(P_\pm(t))_\mathrm{prin}(x,\xi)=P^{(\pm)}_A(x,\xi;t),
\end{equation}
\begin{equation}
\label{pseudodifferential projections equation 2}
(P_\pm(t))^2=P_\pm(t) \mod \Psi^{-\infty},
\end{equation}
\begin{equation}
\label{pseudodifferential projections equation 3}
[A_t,P_\pm(t)]=i\,\partial_tP_\pm(t) \mod \Psi^{-\infty}.
\end{equation}

(b) The pseudodifferential projections $P_\pm(t)$ are uniquely determined, modulo $\Psi^{-\infty}$,  by conditions \eqref{pseudodifferential projections equation 1}--\eqref{pseudodifferential projections equation 3}. 

(c) The pseudodifferential projections $P_\pm(t)$ automatically satisfy
\begin{equation}
\label{pseudodifferential projections equation 4}
P_\pm(t)=(P_\pm(t))^* \mod \Psi^{-\infty},
\end{equation}
\begin{equation}
\label{pseudodifferential projections equation 5}
P_\pm(t)P_\mp(t)=0 \mod \Psi^{-\infty}
\end{equation}
and
\begin{equation}
\label{pseudodifferential projections equation 6}
P_+(t)+P_-(t)=\mathrm{Id} \mod \Psi^{-\infty},
\end{equation}
where $\mathrm{Id}$ is the identity operator on $L^2(\Sigma;\mathbb{C}^2)$.
\end{theorem}

\begin{remark}
The existence of time-dependent pseudodifferential projections satisfying the above conditions for the special case of the Dirac operator was proved by Gérard and Stoskopf in \cite{gerardDirac} via a different strategy,  by employing methods from linear adiabatic theory.  The approach followed here will (a) lead to an explicit algorithm for the construction of the full symbol of our projections and (b) allow one to understand how each condition defining pseudo\-differential projections affects each homogeneous component of the symbol. 
\end{remark}

\begin{proof}
The proof relies on an adaptation to the case at hand of the strategy devised in \cite{part1,diagonalization}.

(a) In order to prove existence, we will determine the structure of the full symbols of $P_\pm(t)$ explicitly.  We will achieve this by constructing a sequence of pseudodifferential operators $P_{\pm,-k}(t)$, $k=0,1,2,\dots$, satisfying 
\begin{equation}
\label{proof projections equation 1}
P_{\pm,-k-1}(t)-P_{\pm,-k}(t) \in C^\infty(\mathbb{R};\Psi^{-k-1}(\Sigma;\mathbb{C}^2)),
\end{equation}
\begin{equation}
\label{proof projections equation 2}
(P_{\pm,-k}(t))^2=P_{\pm,-k}(t)  \mod C^\infty(\mathbb{R};\Psi^{-k-1}(\Sigma;\mathbb{C}^2)),
\end{equation}
\begin{equation}
\label{proof projections equation 3}
[A_t,P_{\pm,-k}(t)]=i\,\partial_tP_{\pm,-k}(t) \mod C^\infty(\mathbb{R};\Psi^{-k-1}(\Sigma;\mathbb{C}^2)),
\end{equation}
for $k=0,1,2,\dots$. 

Let $P_{\pm,0}(t)=(P_{\pm,0}(t))^*\in C^\infty(\mathbb{R};\Psi^0(\Sigma;\mathbb{C}^2))$ be arbitrary pseudodifferential operators satisfying \eqref{pseudodifferential projections equation 1}.  The latter automatically satisfy \eqref{pseudodifferential projections equation 2} and \eqref{pseudodifferential projections equation 3} modulo $ C^\infty(\mathbb{R};\Psi^{-1}(\Sigma;\mathbb{C}^2))$. Let us construct $P_{\pm,-k}(t)$, $k=1,2,\dots$, by solving \eqref{proof projections equation 1}--\eqref{proof projections equation 3} recursively.

To this end,  suppose we have determined $P_{\pm,1-k}(t)$ and put 
\begin{equation}
\label{proof projections equation 4}
P_{\pm,-k}(t):=P_{\pm,1-k}(t)+Q_{\pm,-k}(t), \qquad k=1,2,\dots,
\end{equation}
where $Q_{\pm,-k}(t)=(Q_{\pm,-k}(t))^*\in C^\infty(\mathbb{R};\Psi^{-k}(\Sigma;\mathbb{C}^2))$ is an unknown pseudodifferential operator to be determined.

Formula \eqref{proof projections equation 4} implies that condition \eqref{proof projections equation 1} is automatically satisfied, whereas satisfying \eqref{proof projections equation 2} and \eqref{proof projections equation 3} reduces to solving
\begin{equation}
\label{proof projections equation 5}
\bigl((P_{\pm,1-k}(t)+Q_{\pm,k}(t))^2-P_{\pm,1-k}(t)-Q_{\pm,-k}(t)\bigr)_\mathrm{prin,-k}=0,
\end{equation}
\begin{equation}
\label{proof projections equation 6}
\bigl([A_t,P_{\pm,1-k}(t)+Q_{\pm,-k}(t)]-i\,\partial_t(P_{\pm,1-k}(t)+Q_{\pm,-k}(t))\bigr)_\mathrm{prin,1-k}=0,
\end{equation}
respectively.  Note that \eqref{proof projections equation 5} and \eqref{proof projections equation 6} give us a system of equations for the smooth matrix-function $(Q_{\pm,-k}(t))_\mathrm{prin}$. Also note that the LHS of \eqref{proof projections equation 6} is a pseudodifferential operator of order $1-k$ --- one order higher than the LHS of \eqref{proof projections equation 5} --- because $A_t$ is an operator of order $1$.

More explicitly, the system \eqref{proof projections equation 5}--\eqref{proof projections equation 6} reads (further on we drop the dependence on $t$, when this does not generate confusion,  for the sake of readability)
\begin{equation}
\label{proof projections equation 7}
P^{(\pm)}_A(Q_{\pm,-k})_\mathrm{prin}+(Q_{\pm,-k})_\mathrm{prin}P^{(\pm)}_A-(Q_{\pm,-k})_\mathrm{prin}=R_{\pm,-k},
\end{equation}
\begin{equation}
\label{proof projections equation 8}
[(A_t)_\mathrm{prin},(Q_{\pm,-k})_\mathrm{prin}]=M_{\pm,-k},
\end{equation}
where
\begin{equation*}
\label{proof projections equation 9}
R_{\pm,-k}:=-((P_{\pm,1-k})^2-P_{\pm,1-k})_{\mathrm{prin},-k},
\end{equation*}
\begin{equation*}
\label{proof projections equation 10}
M_{\pm,-k}:=([P_{\pm,1-k},A_t]+i\,\partial_tP_{\pm,1-k})_{\mathrm{prin},1-k}.
\end{equation*}

Direct inspection tells us that \eqref{proof projections equation 7} has a solution only if 
\begin{equation}
\label{proof projections equation 11}
P^{(+)}_AR_{\pm,-k}P^{(-)}_A=0.
\end{equation}
The solvability condition \eqref{proof projections equation 11} is satisfied. Indeed, for the upper choice of sign we have
\begin{multline*}
\label{proof projections equation 12}
P^{(+)}_AR_{+,-k}P^{(-)}_A
\\
=
- (P_{+,1-k})_{\mathrm{prin},0} ((P_{+,-k+1})^2-P_{+,1-k})_{\mathrm{prin},-k}(\mathrm{Id}-P_{+,1-k})_{\mathrm{prin},0}
\\
=
-[P_{+,1-k}((P_{+,1-k})^2-P_{+,1-k})(\mathrm{Id}-P_{+,1-k})]_{\mathrm{prin},-k}
\\
=
[((P_{+,1-k})^2-P_{+,1-k})^2]_{\mathrm{prin},-k}
=
0,
\end{multline*}
where in the last step we used the inductive assumption.  A similar argument yields \eqref{proof projections equation 11} for the lower choice of sign.

Then the general solution of \eqref{proof projections equation 7} reads
\begin{equation}
\label{proof projections equation 13}
(Q_{\pm,-k})_\mathrm{prin}=-R_{\pm,-k}+P^{(\pm)}_AR_{\pm,-k}+R_{\pm,-k}P^{(\pm)}_A+ W_{\pm,-k}+(W_{\pm,-k})^*,
\end{equation}
where $W_{\pm,-k}$ is an arbitrary smooth matrix function satisfying
\begin{equation}
\label{proof projections equation 14}
W_{\pm,-k}=P^{(\pm)}_AW_{\pm,-k}P^{(\mp)}_A,
\end{equation}
as one can easily establish by substituting \eqref{proof projections equation 13} into \eqref{proof projections equation 7}, and using \eqref{proof projections equation 11} and the fact that $R_{\pm,-k}$ is Hermitian.

Let
\begin{equation*}
\label{proof projections equation 15}
Q_{\pm,-k}:=-R_{\pm,-k}+P^{(\pm)}_AR_{\pm,-k}+R_{\pm,-k}P^{(\pm)}_A.
\end{equation*}
Then \eqref{proof projections equation 8} can be equivalently rewritten as
\begin{equation}
\label{proof projections equation 16}
(h^{(\pm)}_A-h^{(\mp)}_A)(W_{\pm,-k}-(W_{\pm,-k})^*)=T_{\pm,-k},
\end{equation}
where 
\begin{equation}
\label{proof projections equation 17}
T_{\pm,-k}:=M_{\pm,-k}+[Q_{\pm,-k},(A_t)_\mathrm{prin}].
\end{equation}

The identity \eqref{proof projections equation 14} implies that \eqref{proof projections equation 16} admits a solution only if
\begin{equation}
\label{proof projections equation 18}
P^{(\pm)}_AT_{\pm,-k}P^{(\pm)}_A=0
\end{equation}
and
\begin{equation*}
\label{proof projections equation 19}
(T_{\pm,-k})^*=-T_{\pm,-k}.
\end{equation*}
It is easy to see that the latter is satisfied.
Let us check that the solvability condition \eqref{proof projections equation 18} is also satisfied. 

Let $\widetilde{Q}_{\pm,-k}=(\widetilde{Q}_{\pm,-k})^*\in C^\infty(\mathbb{R};\Psi^{-k}(\Sigma);\mathbb{C}^2)$ be an arbitrary pseudodifferential operator such that $(\widetilde{Q}_{\pm,-k})_\mathrm{prin}=Q_{\pm,-k}$. Then we have
\begin{equation}
\label{proof projections equation 20}
(P_{\pm,1-k}+\widetilde{Q}_{\pm,-k})^2=P_{\pm,1-k}+\widetilde{Q}_{\pm,-k}  \mod C^\infty(\mathbb{R};\Psi^{-k-1}(\Sigma;\mathbb{C}^2))
\end{equation}
and we can recast \eqref{proof projections equation 17} as
\begin{equation}
\label{proof projections equation 21}
T_{\pm,-k}=([P_{\pm,1-k}+\widetilde{Q}_{\pm,-k},A_t]+i\,\partial_tP_{\pm,1-k})_{\mathrm{prin},1-k}.
\end{equation}
Hence,  \eqref{proof projections equation 21} and \eqref{proof projections equation 20} imply
\begin{multline*}
\label{proof projections equation 22}
P^{(\pm)}_AT_{\pm,-k}P^{(\pm)}_A
\\
=
(P_{\pm,1-k}+\widetilde{Q}_{\pm,-k})_{\mathrm{prin},0}([P_{\pm,1-k}+\widetilde{Q}_{\pm,-k},A_t])_{\mathrm{prin},-k+1}(P_{\pm,1-k}+\widetilde{Q}_{\pm,-k})_{\mathrm{prin},0}
\\
+
(P_{\pm,1-k})_{\mathrm{prin},0}(i\,\partial_tP_{\pm,1-k})_{\mathrm{prin},1-k}(P_{\pm,1-k})_{\mathrm{prin},0}
\\
=
i(P_{\pm,1-k}(\partial_tP_{\pm,1-k})P_{\pm,1-k})_{\mathrm{prin},1-k}
\\
=
i\bigl(P_{\pm,1-k}\partial_t(P_{\pm,1-k}P_{\pm,1-k})-P_{\pm,1-k}P_{\pm,1-k}\partial_t(P_{\pm,1-k})\bigr)_{\mathrm{prin},1-k}
\\
=
0.
\end{multline*}

The unique solution to \eqref{proof projections equation 16} is then
\begin{equation}
\label{proof projections equation 18bis}
W_{\pm,-k}=\dfrac{P^{(\pm)}_AT_{\pm,-k}P^{(\pm)}_A}{(h^{(\pm)}_A-h^{(\mp)}_A)},
\end{equation}
where uniqueness follows from the fact that the homogeneous equation
\begin{equation*}
(h^{(\pm)}_A-h^{(\mp)}_A)(W_{\pm,-k}-(W_{\pm,-k})^*)=0
\end{equation*}
complemented by condition \eqref{proof projections equation 14} admits the trivial solution only.

All in all, we have proved that 
\begin{equation*}
\label{proof projections equation 19bis}
P_\pm(t)\sim P_{\pm,0}+\sum_{k=1}^{+\infty} Q_{\pm,-k}(t),
\end{equation*}
with $Q_{\pm,-k}(t)$ constructed iteratively in accordance with \eqref{proof projections equation 13}, \eqref{proof projections equation 18bis},
satisfy \eqref{pseudodifferential projections equation 1}--\eqref{pseudodifferential projections equation 3},
thus establishing existence.  Smoothness with respect to the parameter $t$ of the symbol follows in a straightforward manner from the smoothness in $t$ of $P^{(\pm)}_A$ and the above construction.

(b) Suppose we have two sets of projections $P_\pm(t)$ and $\widetilde P_\pm(t)$ satisfying \eqref{pseudodifferential projections equation 1}--\eqref{pseudodifferential projections equation 3} and define
\begin{equation*}
\label{proof projections equation 20bis}
L_\pm(t):=P_\pm(t)-\widetilde P_\pm(t).
\end{equation*}
Arguing by contradiction, suppose there exists a positive integer $k$ such that
\begin{equation}
\label{proof projections equation 21bis}
L_\pm(t) \in C^\infty(\mathbb{R};\Psi^{-k}(\Sigma;\mathbb{C}^2))\qquad \text{but}\qquad L_\pm(t) \notin C^\infty(\mathbb{R};\Psi^{-k-1}(\Sigma;\mathbb{C}^2)).
\end{equation}
Formula \eqref{proof projections equation 21bis} implies that
\begin{equation}
\label{proof projections equation 22bis}
(L_\pm(t))_{\mathrm{prin},-k}\ne 0.
\end{equation}
The fact that $P_\pm(t)=L_\pm(t)+\widetilde P_\pm(t)$ satisfies \eqref{pseudodifferential projections equation 1}--\eqref{pseudodifferential projections equation 3} translates into the following set of equations:
\begin{subequations}
\begin{equation}
\label{proof projections equation 23}
P^{(\pm)}_A(L_\pm)_\mathrm{prin}+(L_\pm)_\mathrm{prin}P^{(\pm)}_A-(L_\pm)_\mathrm{prin}=0,
\end{equation}
\begin{equation}
\label{proof projections equation 24}
(A_t)_\mathrm{prin} (L_\pm)_{\mathrm{prin}}-(L_\pm)_{\mathrm{prin}}(A_t)_\mathrm{prin}=0.
\end{equation}
\end{subequations}
Equation \eqref{proof projections equation 23} implies that
\begin{equation}
\label{proof projections equation 25}
(L_\pm)_\mathrm{prin}=P^{(+)}_A(L_\pm)_\mathrm{prin}P^{(-)}_A+P^{(-)}_A(L_\pm)_\mathrm{prin}P^{(+)}_A.
\end{equation}
Substituting \eqref{proof projections equation 25} into \eqref{proof projections equation 24} we get
\begin{equation*}
\label{proof projections equation 26}
(h^{(+)}_A-h^{(-)}_A)(P^{(+)}_A(L_\pm)_\mathrm{prin}P^{(-)}_A-P^{(-)}_A(L_\pm)_\mathrm{prin}P^{(+)}_A)=0.
\end{equation*}
But the latter only admits the trivial solution $(L_\pm)_\mathrm{prin}=0$, which contradicts \eqref{proof projections equation 22bis}.

(c)
Lastly, let us show that \eqref{pseudodifferential projections equation 1}--\eqref{pseudodifferential projections equation 3} imply \eqref{pseudodifferential projections equation 4}--\eqref{pseudodifferential projections equation 6}.

Put
\begin{equation*}
\label{proof projections equation 27}
B_\pm(t)=P_\pm(t)-(P_\pm(t))^*
\end{equation*}
and suppose 
\begin{equation}
\label{proof projections equation 28}
B_\pm(t) \in C^\infty(\mathbb{R};\Psi^{-k}(\Sigma;\mathbb{C}^2))\qquad \text{but}\qquad B_\pm(t) \notin C^\infty(\mathbb{R};\Psi^{-k-1}(\Sigma;\mathbb{C}^2))
\end{equation}
for some $k\in \mathbb{N}$.  Then, arguing as in part (b),  one obtains that $(B_\pm(t))_\mathrm{prin}=0$, because it satisfies \eqref{proof projections equation 23} and \eqref{proof projections equation 24},  thus contradicting \eqref{proof projections equation 28}.
Therefore $B_\pm\in C^\infty(\mathbb{R};\Psi^{-\infty}(\Sigma;\mathbb{C}^2))$, which gives us \eqref{pseudodifferential projections equation 4}.

Next, define
\begin{equation*}
\label{proof projections equation 29}
C_\pm(t):=P_\pm(t)P_\mp(t)
\end{equation*}
and suppose there exists an integer $k$ such that
\begin{equation}
\label{proof projections equation 30}
C_\pm(t) \in C^\infty(\mathbb{R};\Psi^{-k}(\Sigma;\mathbb{C}^2))\qquad \text{but}\qquad C_\pm(t) \notin C^\infty(\mathbb{R};\Psi^{-k-1}(\Sigma;\mathbb{C}^2)).
\end{equation}
Properties \eqref{pseudodifferential projections equation 1}--\eqref{pseudodifferential projections equation 3} imply
\begin{subequations}
\begin{equation}
\label{proof projections equation 31}
P^{(\pm)}_A(C_\pm)_{\mathrm{prin},-k}-(C_\pm)_{\mathrm{prin},-k}=0,
\end{equation}
\begin{equation}
\label{proof projections equation 32}
(C_\pm)_{\mathrm{prin},-k}P^{(\mp)}_A-(C_\pm)_{\mathrm{prin},-k}=0
\end{equation}
and
\begin{equation}
\label{proof projections equation 33}
[A_t, P_\pm P_\mp]=i\,\partial_t (P_\pm P_\mp) \mod C^\infty(\mathbb{R};\Psi^{-\infty}(\Sigma;\mathbb{C}^2)),
\end{equation}
which yields
\begin{equation}
\label{proof projections equation 34}
[(A_t)_\mathrm{prin}, (C_\pm)_{\mathrm{prin},-k}]=0.
\end{equation}
\end{subequations}
The system of equations \eqref{proof projections equation 31}, \eqref{proof projections equation 32} and \eqref{proof projections equation 34} admits only the trivial solution $(C_\pm)_{\mathrm{prin},-k}=0$, thus contradicting \eqref{proof projections equation 30}. Therefore $C_\pm\in C^\infty(\mathbb{R};\Psi^{-\infty}(\Sigma;\mathbb{C}^2))$, which gives us \eqref{pseudodifferential projections equation 5}.

Finally,  put
\begin{equation*}
\label{proof projections equation 35}
F(t):=\mathrm{Id}-P_+(t)-P_-(t)
\end{equation*}
and, arguing by contradiction, suppose there exists an integer $k$ such that
\begin{equation}
\label{proof projections equation 30bis}
F(t) \in C^\infty(\mathbb{R};\Psi^{-k}(\Sigma;\mathbb{C}^2))\qquad \text{but}\qquad J_\pm(t) \notin C^\infty(\mathbb{R};\Psi^{-k-1}(\Sigma;\mathbb{C}^2)).
\end{equation}

Properties \eqref{pseudodifferential projections equation 2} and \eqref{pseudodifferential projections equation 5} imply
\begin{equation*}
\label{proof projections equation 31bis}
(P_++P_-)F\in C^\infty(\mathbb{R};\Psi^{-\infty}(\Sigma;\mathbb{C}^2)),
\end{equation*}
so that
\begin{equation}
\label{proof projections equation 32bis}
((P_++P_-)F)_{\mathrm{prin},-k}=0.
\end{equation}
But
\begin{equation}
\label{proof projections equation 33bis}
((P_++P_-)F)_{\mathrm{prin},-k}=(P_++P_-)_\mathrm{prin}\, F_{\mathrm{prin}, -k}=F_{\mathrm{prin}, -k}.
\end{equation}
Formulae \eqref{proof projections equation 32bis} and \eqref{proof projections equation 33bis} imply $F_{\mathrm{prin}, -k}=0$, which contradicts \eqref{proof projections equation 30bis}.  Therefore $F_\pm\in C^\infty(\mathbb{R};\Psi^{-\infty}(\Sigma;\mathbb{C}^2))$, which gives us \eqref{pseudodifferential projections equation 6}. This concludes the proof.
\end{proof}

Clearly,  one can always choose $P_\pm(t)$ so that condition \eqref{pseudodifferential projections equation 4} is satisfied exactly, not only modulo a smoothing operator, by suitably adjusting the smoothing error. In what follows,  without loss of generality, we will take $P_\pm(t)$ to be self-adjoint.

\begin{remark}
\label{remark spectral projections}
Note that, in general,  the pseudodifferential projections from Theorem~\ref{theorem pseudodifferential projections} for the choice $A_t=\overline \Dir_t$ do \emph{not} coincide with the spectral projections $\theta(\pm \overline{\Dir}_t)$, not even modulo $\Psi^{-\infty}$.  Here $\theta$ is the Heaviside step function. Indeed, it follows from \cite[Theorem~2.7]{part1} that the latter are pseudodifferential operators smoothly depending on $t$ uniquely determined by 
\eqref{pseudodifferential projections equation 1}, \eqref{pseudodifferential projections equation 2} and
\begin{equation}
\label{remark spectral projections equation 1}
[\overline{\Dir}_t, P_{\pm}(t)]=0 \mod \Psi^{-\infty}.
\end{equation}
Compare \eqref{remark spectral projections equation 1} with \eqref{pseudodifferential projections equation 3}.
The two, of course, coincide in the ultrastatic setting, \ie  $g=-dt^2+h$, where the (reduced) Dirac operator and pseudodifferential projections are time-independent.
\end{remark}

\begin{remark}
Observe that Theorem~\ref{theorem pseudodifferential projections} can be effortlessly generalised to families of operators of arbitrary order acting on the trivial $\mathbb{C}^m$-bundle, $m\ge 2$,  over $\Sigma$, so long as their principal symbols have simple eigenvalues.  We formulate the result for operators of order $1$ acting on 2-columns of scalar functions for the sake of clarity, so that we can more easily specialise the statement to the operators $\pm \overline \Dir_t$.
\end{remark}

\subsubsection{Algorithmic construction}
\label{Algorithmic construction}

The arguments from the proof of Theorem~\ref{theorem pseudodifferential projections} can be summarised in the form of a concise algorithm for the construction of the full symbol of the pseudodifferential projections $P_\pm(t)$, given below.

\begin{enumerate}
\item 
Choose arbitrary pseudodifferential operators $P_{\pm,0}(t)$ of order zero satisfying
\[
[P_{\pm,0}(t)]_\mathrm{prin}=P^\pm_A(y,\eta;t).
\]

\item
Assuming we know $P_{\pm,-k+1}(t)$, compute the following quantities:
\begin{enumerate}[(i)]
\item 
$R_{\pm,-k}(t):=-((P_{\pm,-k+1}(t))^2-P_{\pm,-k+1}(t))_{\mathrm{prin},-k}$,

\item
$S_{\pm,-k}(t):=-R_{\pm,-k}(t)+P^\pm_A R_{\pm,-k}(t)+R_{\pm,-k}(t)P^\pm_A $,

\item
$$T_{\pm,-k}(t):=([P_{\pm,-k+1}(t),A_t]+i\partial_t P_{\pm,-k+1}(t))_{\mathrm{prin},-k+1}+[S^{(\pm)}_k(t),(A_t)_\mathrm{prin}].$$
\end{enumerate}
Note that $R_{\pm,-k}(t)$,  $S_{\pm,-k}(t)$ and $T_{\pm,-k}(t)$ are smooth functions of $t$ valued in $2\times 2$ smooth matrix-functions on $\T^*\Sigma\setminus \{0\}$,  positively  homogeneous in momentum of degree $-k$, $-k$ and $-k+1$, respectively.

\item 
Choose pseudodifferential operators $Q_{\pm,k}(t)$ of order $-k$ such that
\begin{equation*}
(Q_{\pm,-k}(t))_\mathrm{prin}=S_{\pm,-k}(t)
+\frac{P^{(+)}_AT_{\pm,-k}(t)P^{(-)}_A-P^{(-)}_AT_{\pm,-k}(t) P^{(+)}_A}{h^{(+)}_A-h^{(-)}_A}
\end{equation*}
and set $P_{\pm,-k}(t):=P_{\pm,-k+1}(t)+Q_{\pm,k}(t)$. 

\item
Then
\[
P_\pm(t)\sim P_{\pm,0}(t)+\sum_{j=1}^{+\infty} Q_{\pm,-k}(t)\,.
\]
\end{enumerate}

\begin{remark}
Note that the above algorithm is independent of the choice of a particular quantisation for our pseudodifferential operators.
\end{remark}

\subsection{Construction of the propagators}
\label{Construction of the propagators}

Our strategy is to construct 
 the  integral (Schwartz) kernel of the positive and negative propagators $U_L^{(\pm)}$
in the form of two oscillatory integrals 
\begin{multline}
\label{oscillatory integral algorithm 1}
\mathcal{I}_{\varphi^\pm}(\mathfrak{a}):=\frac{1}{(2\pi)^{3}}\int_{\T'_y\Sigma} e^{i\varphi^\pm(t,x;s,y,\eta)}\mathfrak{a}^\pm(t;s,y,\eta)
\\
\times \chi^\pm(t,x;s,y,\eta)\,w^\pm(t,x;s,y,\eta)\,d\eta 
\end{multline}
where
\begin{itemize}
\item
the quantities $\mathfrak{a}^\pm \in C^\infty(\mathbb{R};S^0_\mathrm{ph}(\T'\Sigma_s;\mathrm{Mat}(2,\mathbb{C})))$ are polyhomogeneous matrix-valued symbols of order zero on $\T'\Sigma_s$ which depend smoothly on $t$ and $s$ to be determined, 
\begin{equation*}
\label{definition mathfrak a}
\mathfrak{a}^\pm \sim \sum_{k=0}^{+\infty} \mathfrak{a}^\pm_{-k}, \qquad \mathfrak{a}_{-k}(t;s,y,\lambda \eta)=\lambda^{-k} \, \mathfrak{a}_{-k}(t;s,y, \eta) \quad \forall\lambda>0,
\end{equation*}
where $\sim$ stands for asymptotic expansion, see \cite[Section~3.3]{shubin};

\item the functions $\varphi^\pm$ are distinguished geometric complex-valued phase functions defined in Subsection~\ref{The Levi-Civita phase functions};

\item
the functions $\chi^\pm\in C^\infty(\M \times \T'\Sigma_s)$ are infinitely smooth cut-offs satisfying
\begin{enumerate}[(i)]
\item
$\chi^\pm(t,x;s,y,\eta)=0$ on $\{(t,x;s,y,\eta)\ |\ \|\eta\|_{h_s}\le 1/2\},$

\item
$\chi^\pm(t,x;s,y,\eta)=1$ on the intersection of $\{(t,x;s,y,\eta)\ |\ \|\eta\|_{h_s}\ge 1\}$ and a conical neighbourhood of $\{(t,x^\pm(t;s,y,\eta);s,y,\eta)\}$,

\item
$\chi^\pm(t,x;s,y,\lambda \eta)=\chi^\pm(t,x;s,y, \eta)$ for all $\lambda\ge 1$ on $$\{(t,x;s,y,\eta)\ |\ \|\eta\|_{h_s}\ge 1\};$$
\end{enumerate}

\item
the weights $w^\pm$ are defined in accordance with
\begin{equation}
\label{definition weight w pm}
w^\pm(t,x;s,y,\eta):=[\rho_{h_s}(y)\rho_{h_t}(x)]^{-1/2}[\left(\det \varphi^\pm_{x^\mu\eta_\nu} \right)^2]^{1/4},
\end{equation}
where $\rho_{h_s}(y):=\sqrt{\det (h_s)_{\alpha\beta}(y)}$ and the branch of the complex root is chosen in such a way that $w^\pm(s,y;s,y,\eta)=[\rho_{h_s}(y)]^{-1}$.
\end{itemize}

\begin{remark}
Observe that the quantity 
\[
[\left(\det \varphi_{x^\mu\eta_\nu} \right)^2]^{1/4}
\]
is a $1/2$-density in $x$ and a $(-1/2)$-density in $y$, as one can easily establish by straightforward calculations. The powers of the Riemannian density in \eqref{definition weight w pm} are chosen in such a way that 
\eqref{definition weight w pm} is a scalar function in $x$ and a $(-1)$-density in $y$, 
so that
\eqref{oscillatory integral algorithm 1} is an scalar function in both $x$ and $y$.
\end{remark}

\subsubsection{The Levi-Civita phase functions}
\label{The Levi-Civita phase functions}

The first step consists in introducing geometric complex-valued phase functions for \eqref{oscillatory integral algorithm 1}.  Indeed, the adoption of distinguished phase functions will allow us to uniquely determine $\mathfrak{a}^\pm$ as invariantly defined scalar matrix-functions.

\medskip

Let $\iota_t: \Sigma_t \to \M$ be the natural embedding of $\Sigma_t$ into $\M$, so that $\hat h_t=\iota^*_t \hat g$.  Here $\iota^*_t$ denotes the pullback along $\iota_t$. 

Let $Y=(s,y)\in \RR \times \Sigma$ be a point of $\M$. For $\eta\in \T_y^*\Sigma \setminus\{0\}$ we denote by $\eta_+$ the unique future pointing null covector in $\T^*_Y\M$ such that 
$
\iota^*_s \eta_+:=\eta
$.
Similarly, for each given $\eta\in \T_y^*\Sigma \setminus\{0\}$ we denote by $\eta_-$ the unique past pointing null covector in $\T^*_Y\M$ such that 
$
\iota^*_s \eta_+:=\eta
$.
Let us put
\begin{equation*}
\hat{\eta}_\pm:=\dfrac{\eta_\pm}{\|\eta\|_{h_s}},
\end{equation*}
where $\|\eta\|_{h_s}:=\sqrt{h^{\alpha\beta}_s(y)\eta_\alpha\eta_\beta}$. See also Figure~\ref{fig:conoluce}.
\begin{figure}[h!]
\centering
\includegraphics[scale=0.32]{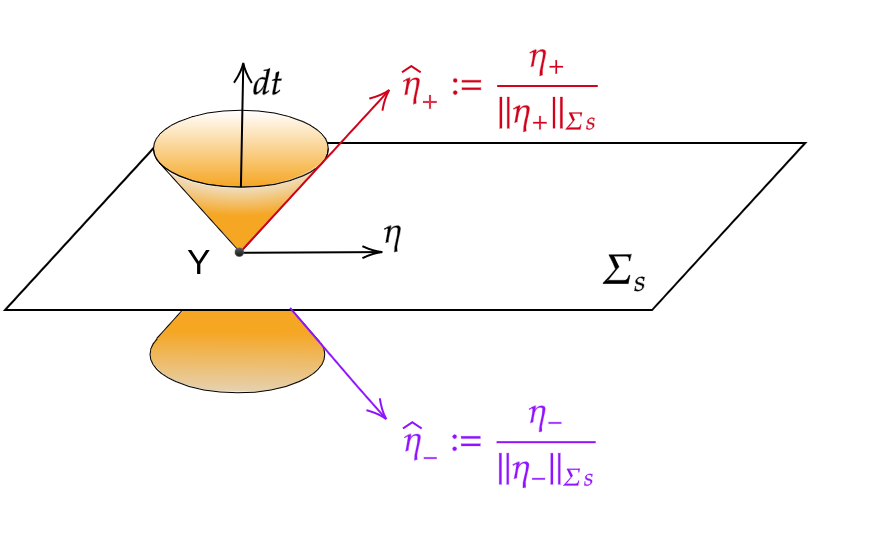}
\caption{Future and past pointing null covectors $\hat \eta_{\pm}$.}
\label{fig:conoluce}
\end{figure}

\begin{definition}[Levi-Civita flow]
\label{definition levi civita flow}
We define the \emph{positive (resp.~negative) Levi-Civita flow} with initial condition $Y=(s,y)\in \M$ to be the map
\begin{equation}
\label{levi civita flow map}
\tau \mapsto (\tilde X^\pm(\tau; s,y,\eta), \tilde \Xi^\pm(\tau;s,y,\eta)),
\end{equation}
where 
\begin{enumerate}[(i)]
\item
$\tau \mapsto \tilde X^\pm(\tau;s,y,\eta)$ is the unique geodesic parameterised by proper time originating at $Y$ with initial cotangent vector $\hat{(\pm \eta)}_\pm$ (see also Figure~\ref{fig:Flow});

\item
$\tilde \Xi^\pm(\tau;s,y,\eta)$ is the parallel transport of $\eta_\pm$ from $Y$ to $\tilde X^\pm(\tau;s,y,\eta)$ along $\tilde X^\pm(\,\cdot\,;s,y,\eta)$.
\end{enumerate}
\end{definition}

By construction, $\tilde X^\pm$ is positively homogeneous in $\eta$ of degree 0 whereas $\tilde \Xi^\pm$ is positively homogeneous in $\eta$ of degree 1, 
\begin{equation}
\label{positive homoegeity Levi Civita flow}
(\tilde X^\pm(\tau;s,y,\lambda\eta),\tilde \Xi^\pm(\tau;s,y,\lambda\eta))=(\tilde X^\pm(\tau;s,y,\eta),\lambda\,\tilde \Xi^\pm(\tau;s,y,\eta)) \quad \forall \lambda>0.
\end{equation}

\begin{figure}[h!]
\centering
\includegraphics[scale=0.28]{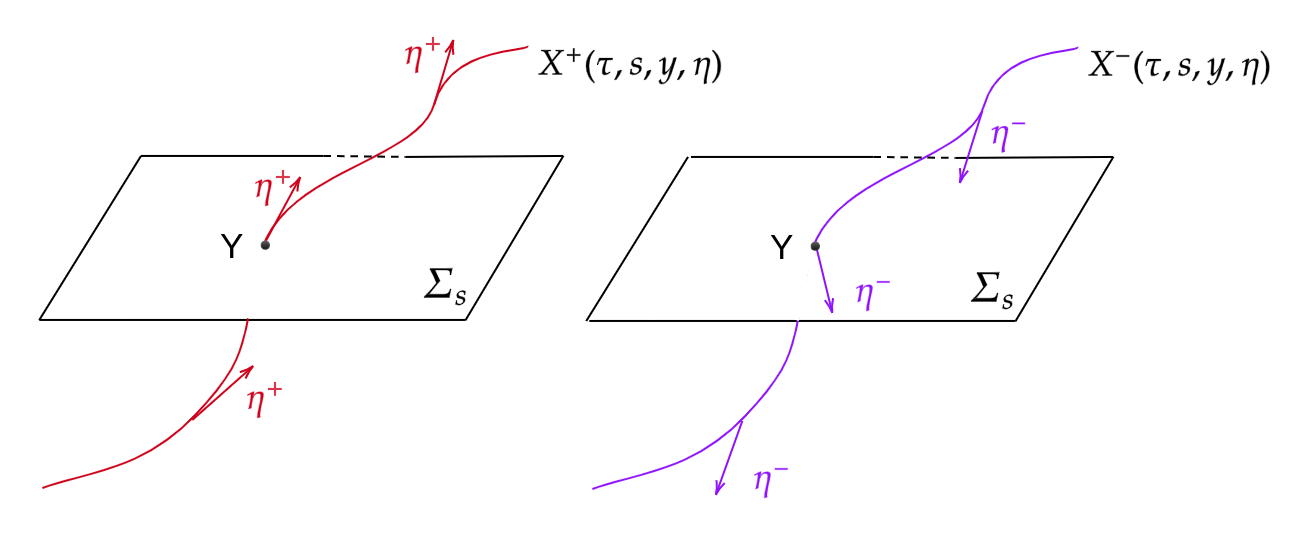}
\caption{Levi Civita flows.}
\label{fig:Flow}
\end{figure}

For practical purposes, it is more convenient to reparameterise the Levi-Civita flow using the global time coordinate $t$ given by Theorem~\ref{thm: Sanchez}, which is always possible. Indeed, if one defines
\[
t^\pm(\tau;s,y,\eta):=t(\tilde X^\pm(\tau;s,y,\eta)),
\]
it is not hard to see that $\eta\in \T_y\Sigma \setminus \{0\}$ implies $\frac{dt^\pm}{d\tau}\ne 0$. We denote by
\[
t\mapsto (X^\pm(t;s,y,\eta),\Xi^\pm(t;s,y,\eta))
\]
the Levi-Civita flow reparameterised by $t$.  Of course, the new parameterisation does not affect property \eqref{positive homoegeity Levi Civita flow}. Furthermore, we have
\begin{equation*}
\label{initial condition Levi Civita flow}
(X^\pm(s;s,y,\eta), \Xi^\pm(s;s,y,\eta))=(Y,\eta_\pm).
\end{equation*}
By means of the Levi-Civita flow, we can define two distinguished phase functions.

\begin{definition}
\label{definition levi civita phase functions}
We call positive ($+$) and negative ($-$) Lorentzian Levi-Civita phase functions the infinitely smooth complex-valued functions
\[
\varphi^\pm: \M \times \RR \times \T'\Sigma \to \mathbb{C}
\]
defined in accordance with 
\begin{multline}
\label{lonrentzian levi civita pm equation}
\varphi^\pm(t,x;s,y,\eta):=-\langle \Xi^\pm(t;s,y,\eta), \left.\operatorname{grad}_Z\sigma(X,Z)\right|_{Z=X^\pm(t;s,y,\eta)} \rangle
\\
+ i\,\|\eta\|_{h_s}\sigma(X,X^\pm(t;s,y,\eta))
\end{multline}
for $X=(t,x)$ in a normal geodesic neighbourhood of $X^\pm(t;s,y,\eta)$ and smoothly continued elsewhere in such a way that $\operatorname{Im} \varphi^\pm \ge 0$ and $\partial_\eta \varphi^\pm\ne0$.  Here $\operatorname{grad}$ stands for gradient and $\sigma$ is the Ruse-Synge world function \cite[Sec.~3]{PPV}.
\end{definition}

The positive and negative Levi-Civita phase functions encode information about propagation of singularities for the reduced Dirac equation \eqref{eq:dirac reduced}.

\begin{proposition}
\label{proposition stationary points of the Levi-Civita phase functions}
Let us denote the critical sets of the Levi-Civita phase functions $\varphi^\pm$ by
\begin{equation*}
\label{critical set phi pm}
\mathcal{C}_{\varphi^\pm}:=\{(t,x;s,y,\eta)\in \M \times\mathbb{R}\times \T^*\Sigma \setminus\{0\}\ |\ \partial_\eta \varphi^\pm(t,x;s,y,\eta)=0\}.
\end{equation*}
Then,  $\mathcal{C}_{\varphi^\pm}$ coincide with the submanifolds $\Phi^\pm\subset \M \times\mathbb{R}\times \T^*\Sigma$ generated by the Levi-Civita flow, namely
\begin{equation*}
\mathcal{C}_{\varphi^\pm}=\Phi^\pm:=\{(t,x^\pm(t;s,y,\eta);s,y,\eta)\ |\ t,s\in \mathbb{R},\ (y,\eta)\in \T^*\Sigma \setminus\{0\}\}.
\end{equation*}
\end{proposition}
\begin{proof}
The proof relies on standard arguments,  see \eg \cite{LSV,SaVa}. The inclusion $\Phi^\pm\subseteq \mathcal{C}_{\varphi^\pm}$ is obtained by differentiating \eqref{phase function zero on flow} with respect to $\eta$, and using \eqref{phase function grad equal xi} and Definition~\eqref{definition levi civita flow}. Now, by performing a Taylor expansion of $\varphi^\pm$ in the variable $x$ around $x=x^\pm$, it is not hard to see that there exist neighbourhoods $\mathcal{U}^\pm$ of $\Phi^\pm$ such that
\[
\mathcal{C}_{\varphi^\pm}\cap (\mathcal{U}\setminus \Phi^\pm)=\emptyset.
\]
The latter and the fact that $\varphi^\pm$ are chosen in such a way that $\partial_\eta \varphi^\pm\ne 0$ for $X$ outside of a geodesic neighbourhood of $X^\pm$ completes the proof.
\end{proof}

The basic properties of the positive and negative Lorentzian Levi-Civita phase functions are summarised by the following Lemma.
\begin{lemma}
\label{lemma properties levi civita phase functions}
\begin{enumerate}[(a)]
\item
The phase functions $\varphi^\pm$ are positively homogeneous in $\eta$ of degree 1:
\[
\varphi^\pm(t,x;s,y,\lambda\eta)=\lambda\,\varphi^\pm(t,x;s,y,\eta)
\]
for every $\lambda>0$.

\item
The positive and negative Lorentzian Levi-Civita phase functions are related as
\[
\varphi^+(t,x;s,y,\eta)=-\overline{\varphi^-(t,x;s,y,-\eta)}.
\]

\item
We have
\begin{equation}
\label{phase function zero on flow}
\varphi^\pm(t,x^\pm;s,y,\eta)=0,
\end{equation}
\begin{equation}
\label{phase function grad equal xi}
\left.\operatorname{grad}_x\varphi^\pm (t,x;s,y,\eta)\right|_{x=x^\pm}=\xi^\pm,
\end{equation}
\begin{equation}
\label{phase function non degenerate}
\left.\det \left(\frac{\partial^2\varphi^\pm(t,x;s,y,\eta)}{\partial x^\alpha\partial \eta_\beta}\right)\right|_{x=x^\pm}\ne 0,
\end{equation}
where $$x^\pm(t;s,y,\eta):=\iota^{-1}_t\left( X^\pm(t;s,y,\eta) \right)$$ and $$\xi^\pm(t;s,y,\eta):=\iota^*_t \left( \Xi^\pm(t;s,y,\eta) \right).$$
\end{enumerate}
\end{lemma}

\begin{proof}
Properties (a) and (b) follow in a straightforward manner from Definitions~\ref{definition levi civita flow} and~\ref{definition levi civita phase functions}.

As to property (c), it is well known \cite[Sec.~4.1]{PPV} that the Ruse-Synge world function satisfies
\begin{equation}
\label{ruse-synge property 1}
\sigma(X,X)=0, \qquad \nabla_{a}\sigma(X,Y)|_{X=Y}=0,
\end{equation}
\begin{equation}
\label{ruse-synge property 2}
\nabla_{a} \nabla_{b} \,\sigma(X,X)=-\left. \nabla_{a} \nabla'_{b} \,\sigma(X,Y)\right|_{X=Y}=g(a,b)(X),
\end{equation}
where the prime in $\nabla'$ indicates that the Levi-Civita connection acts on the second variable.

Hence, \eqref{phase function zero on flow} follows immediately from \eqref{ruse-synge property 1}.

Differentiating \eqref{lonrentzian levi civita pm equation} with respect to $x^\alpha$ and using \eqref{ruse-synge property 1}--\eqref{ruse-synge property 2}, we obtain
\begin{align*}
\left.\partial_{x^\alpha}\varphi^\pm\right|_{x=x^\pm}
&
=-\left.\Xi^\pm_a g^{ab}(X^\pm) \nabla_\alpha\nabla'_b\,\sigma(X,Z)\right|_{Z=X=(t,x^\pm)}
\\
&+ i\,\|\eta\|_{h_s}
\left.\nabla_a\sigma(X,X^\pm(t;s,y,\eta))\right|_{X=(t,x^\pm)}
\\
&
=
\Xi^\pm_a g^{ab}(X^\pm)g_{\alpha b}(X^\pm)=\Xi^\pm_a \delta_\alpha{}^a
=\xi^\pm_\alpha,
\end{align*} 
which gives us \eqref{phase function grad equal xi}.

Finally,  using once again \eqref{ruse-synge property 1} and \eqref{ruse-synge property 2},  a straightforward computation yields
\begin{equation}
\label{phi xx positive definite}
\left. \operatorname{Im} \partial_{x^\alpha x^\beta}\varphi^\pm \right|_{x=x^\pm}=i\,\|\eta\|_{h_s}g_{\alpha\beta}(X^\pm).
\end{equation}
Formula \eqref{phi xx positive definite} tells us that $\left. \operatorname{Im} \partial_{x^\alpha x^\beta}\varphi^\pm \right|_{x=x^\pm}$ is positive definite. Then \eqref{phase function non degenerate} follows from \cite[Corollary~2.4.5]{SaVa} and Proposition~\ref{proposition stationary points of the Levi-Civita phase functions}.
\end{proof}

Our Levi-Civita phase functions $\varphi^\pm$ warrant a number of remarks:
\begin{enumerate}
\item 
 $\varphi^\pm$ are \emph{complex-valued}, as opposed to real-valued, as customary in the classical constructions of hyperbolic parametrices.  This is crucial in ensuring,  in view of the results from \cite{LSV} (see also \cite{wave,lorentzian,dirac,review}), that we can represent the integral kernel of $U^{(\pm)}(t;s)$ as a single oscillatory integral globally in spacetime.  In particular,  the complexity ensures that condition \eqref{phase function non degenerate} holds for all values of the argument.  Real-valued phase functions fail to satisfy \eqref{phase function non degenerate} in the presence of caustics.

\item
$\varphi^\pm$ are geometric in nature, in that they are fully specified by the Lorentzian metric structure, at least in a neighbourhood of the Levi-Civita flow.

\item
The way one continues $\varphi^\pm$ outside of such neighbourhood does not affect the singular part of the oscillatory integral. Different choices of smooth continuations result in an error $=0 \mod \Psi^{-\infty}$,  as one can show by a standard
(non)stationary phase argument.
\end{enumerate}

\begin{remark}
Observe that in the Riemannian setting --- namely, when $\hat \Dir_t$ is independent of $t$,  in other words, the ultrastatic case --- the flow \eqref{levi civita flow map} is nothing but Hamiltonian flow of $\pm(\sqrt{-\Delta})_\mathrm{prin}$, which satisfies
\[
x^-(t;y,\eta)=x^+(t;y,-\eta),
\]
\[
\xi^-(t;y,\eta)=-\xi^+(t;y,-\eta).
\]
Note that a Riemannian version of our Levi-Civita phase functions were used in \cite{wave} and \cite{dirac}, whereas an analogue of $\varphi^+$ in the Lorentzian setting appeared in \cite{lorentzian}.
\end{remark}


\subsubsection{The algorithm}\label{algorithm}

Our next task is to write down an algorithm to determine the matrix-functions $\mathfrak{a}^\pm$ so that the oscillatory integrals $\mathcal{I}_{\varphi^\pm}(\mathfrak{a})$ defined in~\eqref{oscillatory integral algorithm 1} are the (Schwartz) kernel of $U_L^{(\pm)}$.
\medskip

{\bf Step 1}.  Set $\chi^\pm $ equal to $0$ and act with the operator $-i\partial_t+\overline \Dir_t$ on the oscillatory integrals~\eqref{oscillatory integral algorithm 1}. This produces new oscillatory integrals
$
\mathcal{I}_{\varphi^\pm}(b^\pm),
$
where
\begin{equation}
\label{amplitude bpm}
b^\pm(t,x;s,y,\eta)=e^{-i\varphi^\pm}[w^\pm]^{-1}
(-i\partial_t+\overline \Dir_t)\left(
e^{i\varphi^\pm}\mathfrak{a}^\pm\,w^\pm
\right).
\end{equation}
Note that property \eqref{phase function non degenerate} ensures that $w^\pm$ do not vanish in a neighbourhood of the Levi-Civita flow.

\

{\bf Step 2}.  The matrix-functions \eqref{amplitude bpm}, unlike the symbols $\mathfrak{a}^\pm$,  depend on the variable $x$. Step 2 consists in removing the dependence of $b^\pm$ on $x$ by means of a procedure known as \emph{reduction of the amplitude}.

To begin with, let us observe that one can equivalently recast the phase functions $\varphi^\pm$ using only ``Cauchy surface information'' as
\begin{multline}
\label{lonrentzian levi civita pm equation equivalent}
\varphi^\pm(t,x;s,y,\eta)=-\frac12\langle \xi^\pm(t;s,y,\eta), \left.\operatorname{grad}_z\operatorname{dist}^2_{\Sigma_t}(x,z)\right|_{z=x^\pm(t;s,y,\eta)} \rangle
\\
+ \frac i2\,\|\eta\|_{h_s}\operatorname{dist}^2_{\Sigma_t}(x,x^\pm(t;s,y,\eta)),
\end{multline}
compare with \eqref{lonrentzian levi civita pm equation}. Here the expression \eqref{lonrentzian levi civita pm equation equivalent} holds for $x$ in a neighbourhood of $x^\pm(t;s,y,\eta)$.  The essential idea to exclude the dependence on $x$ is to expand $b^\pm$ in power series in $x$ about $x=x^\pm$ and integrate by parts.  One has
\begin{equation}
\label{expansion for b}
b^\pm=\left. b^\pm \right|_{x=x^\pm}+(x-x^\pm)^\alpha c^\pm_\alpha=\left. b^\pm \right|_{x=x^\pm}+\partial_{\eta^\alpha}\varphi^\pm \,\tilde c^\pm_\alpha,
\end{equation}
for some covectors $c^\pm=c^\pm(t,x;y,\eta)$ and $\tilde c^\pm=\tilde c^\pm(t,x;y,\eta)$.  In writing \eqref{expansion for b} we are using the fact that $\left.\partial_\eta \varphi^\pm\right|_{x=x^\pm}=0$, which follows from \eqref{phase function zero on flow}. Of course, $\tilde c^\pm$ can be written explicitly in terms of $c^\pm$ and $\varphi^\pm$.
Plugging \eqref{expansion for b} into the expression for $\mathcal{I}_{\varphi^\pm}(b^\pm)$, and formally integrating by parts one obtains
\begin{multline}
\label{amplitude reduction euristics}
\mathcal{I}_{\varphi^\pm}(b^\pm)=\frac{1}{(2\pi)^{3}}\int_{\T'_y\Sigma} e^{i\varphi^\pm} b^\pm \bigg|_{x=x^\pm} \hspace{-4mm} w^\pm \,d\eta \\
 +\frac{i}{(2\pi)^{3}}\int_{\T'_y\Sigma} e^{i\varphi^\pm} (w^\pm)^{-1} \left( \frac{\partial}{\partial \eta_\alpha} \tilde c^\pm_\alpha\, w^\pm \right) w^\pm\,d\eta.
\end{multline}
Now, the the amplitude of the first integral in the RHS of \eqref{amplitude reduction euristics} is $x$-independent, whereas that of the second integral is $x$-dependent but of one degree lower as a polyhomogeneous symbol. Repeated iterations of the above argument allow one to turn $\mathcal{I}_{\varphi^\pm}(b^\pm)$ into an oscillatory integral with $x$-independent amplitude plus an oscillatory integral whose amplitude is a symbol of order $-\infty$. 

More precisely,  put
\begin{equation*}
L_\alpha^{\pm}:=\left[(\partial^2_{x\eta}\varphi^\pm)^{-1}\right]_\alpha{}^\beta \dfrac{\partial}{\partial x^\beta}
\end{equation*}
and define
\begin{subequations}\label{operators mathfrak S with j}
\begin{gather*}
\mathfrak{S}_0^{\pm}:=\left.\left( \,\cdot\, \right)\right|_{x=x^\pm(t;s,y,\eta)}\,, \label{mathfrak S0 with j}\\
\mathfrak{S}_{-k}^{\pm}:=\mathfrak{S}_0^{\pm} \left[ i \, [w^\pm]^{-1} \frac{\partial}{\partial \eta_\beta}\, w^{\pm} \left( 1+ \sum_{1\leq |\boldsymbol{\alpha}|\leq 2k-1} \dfrac{(-\partial_\eta\varphi^\pm)^{\boldsymbol{\alpha}}}{\boldsymbol{\alpha}!\,(|\boldsymbol{\alpha}|+1)}\,L^{\pm}_{\boldsymbol{\alpha}} \right) L^{\pm}_\beta  \right]^k\,,\label{mathfrak Sk with j}
\end{gather*}
\end{subequations}
where $\boldsymbol{\alpha}\in \mathbb{N}^3$, $|\boldsymbol{\alpha}|=\sum_{j=1}^3 \alpha_j$ and $(-\partial_\eta\varphi^{\pm})^{\boldsymbol{\alpha}}:=(-1)^{|\boldsymbol{\alpha}|}\, \prod_{j=1}^3(\partial_{\eta_j}\varphi^\pm)^{\alpha_j}$. When acting on a function positively homogeneous in momentum, the operators $\mathfrak{S}^\pm_{-k}$ excludes the dependence on $x$ whilst decreasing the degree of homogeneity by $k$.

The {\emph{amplitude-to-symbol operators}
\begin{equation*}
\mathfrak{S}^\pm:=\sum_{k=0}^{+\infty} \mathfrak{S}^\pm_k
\end{equation*}
map the $x$-dependent amplitudes $\mathfrak{b}^\pm(t,x;s,y,\eta)$ to the $x$-independent symbols
\begin{equation*}
\mathfrak{b}^\pm(t;s,y,\eta)\sim \sum_{j=-1}^{+\infty} \mathfrak{b}^\pm_{-j}(t;s,y,\eta), \qquad \mathfrak{b}^\pm_{-j}:=\sum_{k+s=j} \mathfrak{S}^{\pm}_{-k}b^\pm_{-s},
\end{equation*}
and we have
\begin{equation}
\label{identity amplitude to symbol oscillatory integrals}
\mathcal{I}_{\varphi^\pm}(b^\pm)=\mathcal{I}_{\varphi^\pm}(\mathfrak b^\pm) \mod C^\infty,
\end{equation}
where $\mod C^\infty$ means that LHS and RHS only differ by an infinitely smooth function.
We refer the reader to \cite[Appendix~A]{wave} for detailed exposition and proofs, up to appropriate straightforward adaptations.

\

{\bf Step 3}.  Equations \eqref{amplitude bpm} and \eqref{identity amplitude to symbol oscillatory integrals} imply that for \eqref{oscillatory integral algorithm 1} to be the Schwartz kernel of $U(t,s)$ (recall \eqref{propagator full}) one needs to have
\begin{equation}
\label{transport equation formula 0}
\mathfrak{b}^\pm(t;s,y,\eta)=0.
\end{equation}
Imposing condition \eqref{transport equation formula 0} degree of homogeneity by degree of homogeneity results in a hierarchy of transport equations --- matrix ordinary differential equations in the variable $t$ for the homogeneous components (matrix functions)  $\mathfrak{a}^\pm_{0}, \mathfrak{a}^\pm_{-1}, \ldots,$ of $\mathfrak{a}^\pm$. More explicitly, these equations read
\begin{subequations}
\begin{equation}
\label{zeroth transport equation}
\left.b^\pm_1\right|_{x=x^\pm(t;s,y,\eta)}=0\,,
\end{equation}
\begin{equation}
\label{first transport equation}
\mathfrak{S}^\pm_{-1}b^\pm_{1}+\mathfrak{S}^\pm_0 b^\pm_0=0\,,
\end{equation}
\begin{equation}
\label{second transport equation}
\mathfrak{S}^\pm_{-2}b^\pm_{1}+\mathfrak{S}^\pm_{-1} b^\pm_0+\mathfrak{S}^\pm_{0} b^\pm_{-1}=0\,,
\end{equation}
\begin{equation*}
\label{transport equation dots}
\dots\,.
\end{equation*}
\end{subequations}
We call \eqref{zeroth transport equation} the \emph{zeroth} transport equation,  \eqref{first transport equation} the \emph{first} transport equation,  \eqref{second transport equation} the \emph{second} transport equations, and so on.  Initial conditions for our transport equations are a delicate matter, as explained in Subsections~\ref{sec:Cauchy ev op} and~\ref{sec:Lor pdo proj}, and are obtained by requiring that
\begin{equation*}
U_L^{(\pm)}(s;s)=P_\pm(s) \mod \Psi^{-\infty},
\end{equation*}
where $P_\pm(s)$ are the pseudodifferential operators given by Theorem~\ref{theorem pseudodifferential projections} for $A_t=\overline \Dir_t$.

\

It only remains to reconcile the algorithmic construction of $P_\pm$ from Subsection~\ref{Algorithmic construction} with our particular representation of pseudodifferential operators $U_L^{(\pm)}(s;s)$ as oscillatory integrals of the form 
\begin{equation*}
\label{particular representation of U(s,s)}
U_L^{(\pm)}(s;s)=\int_{\Sigma} \left.\mathcal{I}_{\varphi^\pm}(\mathfrak{a}^\pm) \right|_{t=s}\rho_{h_s}(y)\,dy
\mod \Psi^{-\infty}\,.
\end{equation*}

On the one hand, in an arbitrary coordinate system, the same for $x$ and $y$, we can represent the Schwartz kernel of $P_\pm(s)$ as
\begin{equation}
\label{initial condition expansion 5}
\frac{1}{(2\pi)^3}\int_{\T'_y\Sigma_s} e^{i(x-y)^\alpha\eta_\alpha} \,\mathfrak{u}^\pm(s,y,\eta)\,d\eta
\end{equation}
where $\mathfrak{u}^\pm$ are determined via the algorithm given in Subsection~\ref{Algorithmic construction} for the choice of the right quantisation.  Observe that $\mathfrak{u}^\pm$ are \emph{not} invariantly defined, but depend on the choice of local coordinates.

On the other hand, our propagator construction involves representing the Schwartz kernel of $U_L^{(\pm)}(s;s)$ as
\begin{equation}
\label{integral kernel U(s;s)}
\frac{1}{(2\pi)^{3}}\int_{\T'_y\Sigma_s} e^{i\phi(x;s,y,\eta)} \mathfrak{a}^\pm(0;s,y,\eta)\, \chi_0(x;s,y,\eta)\,w_0(x;s,y,\eta)\,d\eta,
\end{equation}
where
\begin{itemize}
\item
the function $\phi(x;s,y,\eta):=\varphi^+(0,x;s,y,\eta)=\varphi^-(0,x;s,y,\eta)$ is the \emph{time-independent} Levi-Civita phase function,

\item
$\chi_0$ is a smooth cut-off localising the integration in a neighbourhood of the diagonal and away from the zero section

\item
and
\[
w_0(x;s,y,\eta):=[\rho_{h_s}(y)\rho_{h_s}(x)]^{-1/2}[\det{}^2 \,\partial_{x^\mu}\partial_{\eta_\nu}\phi(x;s,y,\eta)]^{1/4}.
\] 
Here the branch of the complex root is chosen in such a way that $w_0(y;s,y,\eta)=[\rho_{h_s}(y)]^{-1}$.
\end{itemize}

Working in the coordinate system chosen above, the same for $x$ and $y$, we have
\begin{equation*}
\label{initial condition expansion 1}
\phi(x;s,y,\eta)=(x-y)^\alpha\eta_\alpha+O(|x-y|^2)
\end{equation*}
as $\operatorname{dist}_{\Sigma_s}(x,y)\to 0$. Here and further on in this section tensor indices are raised and lowered with respect to the metric $h_s$. Hence,
\begin{equation}
\label{initial condition expansion 2}
e^{i\phi(x;s,y,\eta)}=e^{i(x-y)^\alpha\eta_\alpha}(1+O(|x-y|^2))
\end{equation}
and
\begin{equation}
\label{initial condition expansion 3}
w_0(x;s,y,\eta)=1+O(|x-y|).
\end{equation}
Substituting \eqref{initial condition expansion 2} and \eqref{initial condition expansion 3} into \eqref{integral kernel U(s;s)} and setting $\chi_0=1$, we obtain
\begin{equation*}
\label{integral kernel U(s;s) expanded}
\frac{1}{(2\pi)^{3}}\int_{\T'_y\Sigma_s} e^{i(x-y)^\alpha\eta_\alpha} \mathfrak{a}^\pm(0;s,y,\eta) (1+r_0(x;s,y,\eta))\,d\eta,
\end{equation*}
where $r_0(x;s,y,\eta)=O(|x-y|)$.  Excluding the $x$-dependence in the amplitude of \eqref{integral kernel U(s;s) expanded} by means of the operator
\begin{equation*}
\label{simplified amplitude to symbol initial condition}
\mathcal{S}(\cdot):=\left.\left[\exp \left(i \dfrac{\partial^2}{\partial x^\alpha\partial \eta_\alpha}\right) \right]\right|_{x=y}
\end{equation*}
we arrive at 
\begin{multline}
\label{integral kernel U(s;s) expanded reduced}
\frac{1}{(2\pi)^{3}}\int_{\T'_y\Sigma_s} e^{i\phi(x;s,y,\eta)} \mathfrak{a}^\pm(0;s,y,\eta)\, \chi_0(x;s,y,\eta)\,w_0(x;s,y,\eta)\,d\eta
\\
=\frac{1}{(2\pi)^{3}}\int_{\T'_y\Sigma_s} e^{i(x-y)^\alpha\eta_\alpha}\, \tilde{\mathfrak{a}}^{\pm}(s,y,\eta) \,d\eta \mod C^\infty
\end{multline} 
where the asymptotic expansions of $\tilde{\mathfrak{a}}^{\pm}(s,y,\eta)$ and $\mathfrak{a}^{\pm}(0;s,y,\eta)$ are related as
\begin{equation*}
\label{initial condition expansion 4}
\tilde{\mathfrak{a}}^{\pm}_{-k}=\left.\mathfrak{a}^{\pm}_{-k}\right|_{t=0}+\text{terms involving derivatives in }\eta\text{ of }\left.\mathfrak{a}^\pm_{-0}\right|_{t=0}, \ldots, \left.\mathfrak{a}^\pm_{-k+1}\right|_{t=0}.
\end{equation*}
Then,  by comparing \eqref{initial condition expansion 5} and \eqref{integral kernel U(s;s) expanded reduced}, we see that our initial conditions $\mathfrak{a}^{\pm}_{-k}(0;s,y,\eta)$, $k=0,1,2,\ldots$, are determined algebraically by iteratively imposing
\begin{equation*}
\label{initial condition expansion 6}
\tilde{\mathfrak{a}}^{\pm}_{-k}(s,y,\eta)=\mathfrak{u}^\pm_{-k}(s,y,\eta), \qquad k=0,1,2,\ldots.
\end{equation*}

\begin{remark}
Even though the intermediate steps depend on the choice of local coordinates, the final outcome are invariantly defined smooth matrix functions $\mathfrak{a}^{\pm}_{-k}(0;s,y,\eta)$. This is an important advantage of representing our pseudodifferential operators in the form \eqref{integral kernel U(s;s)}. See also \cite[Section~4]{dirac} for further discussions on the matter.
\end{remark}

\section{Hadamard states and Feynman propagators}\label{sec:appl}

In this last section we will exploit the results from the previous sections to construct Hadamard states for quantum Dirac fields.  This will be done within the framework of the so-called {\em algebraic approach to quantum field theory} (AQFT), in which  the quantisation of a free field theory is realised as a two-step procedure.
First, one assigns to a classical physical system (described either by the solutions space or by the space of classical observables)
 a unital $*$-algebra $\mathcal{A}$, whose elements are interpreted as observables of the system at hand.
Then,  one determines the admissible physical states of the system by identifying a suitable subclass of the linear, positive and normalised functionals $\omega\colon\mathcal{A}\to\mathbb{C}$ on $\mathcal A$.


\subsection{The algebra of Dirac solutions}\label{sec:Dir alg sol}
Let us endow $\sol_{sc }(\overline\Dir_\M)$,  the space of spacelike compact solutions of the reduced Dirac operator $\overline\Dir_\M$ for the metric $g$ (\cf Proposition~\ref{prop:Green}),
with the positive definite Hermitian scalar product
\begin{align}\label{eq:Herm prod}
 \scalar{\cdot}{\cdot}= \int_{\Sigma}\fiber{{\rm tr}_\Sigma\cdot}{\gamma_\M(e_0){\rm tr}_\Sigma\cdot} d\vol_{\Sigma}\,, 
 \end{align}
where ${\rm tr}_\Sigma$ is the map which assigns to a given solution the corresponding initial datum on $\Sigma$,  $\gamma_\M(e_0)$ is the Clifford multiplication by the global vector field $e_0$ defined in accordance with \eqref{Eq: Sigma unit normal}.
One can show that the scalar product~\eqref{eq:Herm prod} does not depend on the choice of the Cauchy hypersurface $\Sigma$ --- see  \eg~\cite[Lemma 3.17]{CQF1}.

As in Section~\ref{sec:Dirac}, let us denote the adjuction map by $\Upsilon$ and set
\begin{align}\label{Hilb sol}
	\scrH^\oplus := \overline{\Big( \sol_{sc }(\overline\Dir_\M)\oplus \Upsilon\sol_{sc }(\overline\Dir_\M) , \scalar{\cdot}{\cdot}_{\oplus}\Big)}^{\| \cdot\|_{\oplus}}
	\,,
\end{align}
where $\scalar{\cdot}{\cdot}_{\oplus}$ is the natural scalar product on the direct sum induced by~\eqref{eq:Herm prod}, and $\|\cdot\|_{\oplus}$ is norm induced by $(\,,\,)_{\oplus}$.  
Moreover, let $\Theta\colon \scrH^\oplus\to \scrH^\oplus$ be the antilinear involution defined by 
\begin{equation}\label{invol}
\Theta(\Phi_1\oplus\Upsilon\Phi_2):=(\Phi_2)\oplus\Upsilon\Phi_1\,.
\end{equation}

\begin{definition}\label{def:alg Dirac sol}
The \textit{algebra of Dirac solutions} is the unital complex $*$-algebra $\mathcal{A}$ freely generated by the abstract quantities $\Xi(\Phi)$, $\Phi\in \scrH^\oplus$,  and the unit $1_{\mathcal{A}}$, together with the following relations:
\begin{itemize}
\item[(i)] Linearity: $\Xi(\alpha \Phi_1 + \beta \Phi_2) =\alpha \Xi(\Phi_1) + \beta\Xi(\Phi_2)$,
\item[(ii)] Hermiticity: $\Xi(\Phi_1)^*=\Xi(\Theta \Phi_1)$,
\item[(iii)] Canonical anti-commutation relations (CARs):
$$ 
\Xi(\Phi_1) \cdot\Xi(\Phi_2) + \Xi(\Phi_2)\cdot \Xi(\Phi_1) =0,
$$ 
$$
\Xi(\Phi_1) \cdot\Xi(\Phi_2)^* + \Xi(\Phi_2)^*\cdot \Xi(\Phi_1) = \scalar{\Phi_1}{\Phi_2} \,1_{\mathcal{A}} \,,
$$
\end{itemize}
for all $\Phi_1,\Phi_2\in \scrH^\oplus$ and $\alpha,\beta\in\CC$.
\end{definition}
In fact, $\mathcal{A}$ can be completed in a unique way to a $C^*$-algebra \cite{araki},  with $C^*$-norm induced by the Hilbert structure of $ \scrH^\oplus$.
With slight abuse of notation,  we shall henceforth regard $\mathcal{A}$ as a $C^*$-algebra.

\begin{remark}\label{rmk:algObs}
We should mention that the algebra of Dirac solutions $\mathcal{A}$ cannot, strictly speaking,  be considered an algebra of observables,  as spacelike separated observables are required to commute and elements of  $\mathcal{A}$ do not fulfil such requirement.  A good candidate for an algebra of observables is the subalgebras  $\mathcal{A}_{\text{obs}} \subset \mathcal{A}$ composed by even elements, namely the subalgebra formed by  linear combinations of products of an even number of generators,  which are invariant under the action of $\Spin_0(1,n)$ (extended to $\mathcal{A}$).
We refer the reader to~\cite{DHP} for further details.
\end{remark}

\subsection{Quasifree Hadamard states}\label{sec:Hadam}

The second step in the quantisation of a free field theory consists in the identification of (algebraic) states.  Once that a state is specified, the Gelfand--Naimark--Segal (GNS) construction guarantees the existence of a representation of the quantum field algebra
as (in general,  unbounded) operators defined in a common dense subspace of some Hilbert
space. We will not worry here about the explicit construction of such representation, but limit ourselves to recall some basic definitions needed later on (see~\cite{IgorValter} for
a general discussion also pointing to several open questions).

\begin{definition}
Given a complex unital $*$-algebra $\mathcal{A}$ we call  \textit{(algebraic) state} any linear functional $\omega:\mathcal A \to \CC$ that is positive, \ie $\omega(\aa^*\aa)\geq 0$ for any $\aa\in\mathcal A $, and normalised, \ie $\omega(1_\mathcal{A})=1$. 
\end{definition}

Since a generic element of the algebra of Dirac solutions can be written as a polynomial in the generators,  in order to specify a state it suffices
to prescribe its action on monomials,  the so-called {\it $n$-point functions}:
\begin{equation}
\label{n point function}
\omega_n(\f_1,\dots, \f_n):= \omega(\Xi(\f_1)\cdots\Xi(\f_n))\,.
\end{equation}

In this paper, we restrict our attention to the subclass of so-called \emph{quasifree} states,  fully determined by their 2-point distributions.

\begin{definition}\label{def:quasifree}
A state $\omega$ on $\mathcal{A}$ is \textit{quasifree} if its $n$-point functions satisfy
	\begin{align*}
		\omega_n(\f_1,\dots,\f_n)
		=\begin{cases}
			0 & n\textrm{ odd}\\
			\sum\limits_{\sigma \in S'_n} (-1)^{\text{\rm{sign}}(\sigma)} \prod\limits_{i=1}^{n/2}
			\;\omega_2(\f_{\sigma(2i-1)},
			\f_{\sigma(2i)}) & n \textrm{ even}
		\end{cases}\,,
	\end{align*}
	where~$S'_n$ denotes the set of ordered permutations of $n$ elements.
\end{definition}


Unlike a free quantum field theory in Minkowski spacetime, where the unique Poincaré-invariant state -- known as Minkowski vacuum -- stands out as a distinguished element in the space of all states, on a general curved
spacetime, which may not have (geometric) symmetries at all,  there is no clear way of identifying a natural state.
A widely accepted criterion to select physically meaningful states is the \textit{Hadamard condition} \cite{Radzikowski-96,Sahlmann-Verch-01}.
The latter, among other useful properties, ensures the finiteness of the quantum fluctuations of the expectation
value of observable. Furthermore,  it allows one to construct Wick polynomials \cite{HW1} and other observable quantities (such as the stress energy tensor) by means of a covariant
scheme \cite{HM}, encompassing a locally covariant ultraviolet renormalization \cite{HW} (see also \cite{IgorValter} for a recent pedagogical review).
These states have also been employed in the study of the Blackhole radiation \cite{gerardUnruh, MP}, in cosmological models  \cite{DFMR},  in applications to spacetime models \cite{FMR1,FMR2},  to name but a few examples.

 For the sake of convenience, we recall below the Hadamard condition in the form a microlocal condition on the wavefront set of the 2-point distribution after \cite{Radzikowski-96},  rather than the 
equivalent geometric version based on the (local) Hadamard parametrix \cite{review,M}.

\

For the remainder of the paper,  $\omega$ will denote a quasifree states on the algebra of Dirac solutions $\mathcal{A}$.
The map $ (\Phi_1,\Phi_2) \mapsto \omega_2(\Phi_1,\Phi_2)$ can be extended by linearity to the space of finite  linear combinations of sections $\Phi_1 \otimes  \Phi_2 \in
\Gamma_c\big((\S\M|_\Sigma\oplus\S^*\M|_\Sigma)\otimes(\S\M|_\Sigma\oplus\S^*\M|_\Sigma )\big)$.
If we impose continuity with respect to the usual topology on the space of compactly supported test sections  we can uniquely extend $2$-point function to a  distribution 
in  $\Gamma'_c\big((\S\M|_\Sigma\oplus\S^*\M|_\Sigma )\otimes (\S\M|_\Sigma\oplus\S^*\M|_\Sigma ) \big)$ which we shall hereafter denote by the same symbol $\omega_2$.\\
Any quasifree state ${\omega}:\mathcal{A}\to\CC$ is defined by its {\em Cauchy surface covariances} $\lambda^\pm$ via the identities
\begin{equation}
\label{relation hadamard state and cauchy surface covariances}
\omega_2(\Phi_1, \Theta\Phi_2)=\lambda^+(\Phi_1, \Phi_2), \qquad  \omega_2(\Theta\Phi_2, \Phi_1)=\lambda^-(\Phi_1, \Phi_2),
\end{equation}
where $\lambda^\pm$ are bidistribution satisfying
\begin{itemize}
\item[(i)]    $\lambda^+(\Phi_1, \Phi_2) + \lambda^-( \Phi_2, \Phi_1)=\scalar{\Phi_1}{\Phi_2}_\oplus 1_{\mathcal{A}}$ ,
\item[(ii)]    $\lambda^\pm(\Phi_1, \Phi_1) \geq 0$,
\end{itemize}
for all $\Phi_1,\Phi_2 \in\scrH^\oplus$.
\medskip

\begin{definition}\label{prop:equiv states}
Given Cauchy surface covariances $\lambda^\pm$ on the algebra of Dirac solutions,  we call \emph{spacetime covariances} the bidistributions $\Lambda^\pm \in \Gamma'_c\big((\S\M\oplus\S^*\M )\otimes (\S\M\oplus\S^*\M ) \big)$ defined as
$$ \Lambda^\pm(\f_1, \f_2) := \lambda^\pm( \Phi_1, \Phi_2), $$
where $\Phi_1$, $\Phi_2$ are the unique sections of the spinor bundle satisfying $\Phi_1=\overline\G^\oplus \f_1$, $\Phi_2=\overline\G^\oplus \f_2$,  $\overline\G^\oplus:=\overline\G\oplus \overline\G\Upsilon^{-1}$,  and $\overline \G$ is the causal propagator for the reduced Dirac operator $\overline \Dir_M$.
\end{definition}

%
%

We are finally in a position to state the Hadamard condition for $\omega$ in terms of a wavefront set condition for the corresponding spacetime covariances.  This formulation was introduced in~\cite{gerard0} and it is equivalent to Definition~\ref{def: Hadamard state orig}.  We adopt here the standard convection that the wavefront set of a vector-valued distribution is  the union of the wavefront sets of its components in an arbitrary but fixed local frame.

\begin{definition}\label{def: Hadamard state}
 A bidistribution $\omega_2 \in \Gamma_{c}'\big((\S\M|_\Sigma\oplus \S^*\M|_\Sigma)\otimes (\S\M|_\Sigma\oplus \S^*\M|_\Sigma)\big)$ is called of \textit{Hadamard form} if and only if the associated spacetime covariances $\Lambda^\pm$ (see Definition~\ref{prop:equiv states} and equation~\eqref{relation hadamard state and cauchy surface covariances}) have the following wavefront sets:
$$\operatorname{WF}(\Lambda^\pm)=\{(X,Y,k_X,-k_Y)\in \T^*(\M\times\M)\backslash\{0\}|\ (X,k_X)\sim(Y,k_Y),\ \pm k_X\rhd 0\},
$$
where $(X,k_X)\sim(Y,k_Y)$ means that $X$ and $Y$ are connected by a lightlike geodesic and $k_Y$ is the parallel transport of $k_X$ from $X$ to $Y$ along said geodesic, whereas $\pm k_X\rhd 0$ means that the covector $\pm k_X$ is future pointing.
 \end{definition}

\subsection{Construction of Hadamard states}
\label{sec:proof1}

An abstract characterisation of quasifree states on a generic CAR algebra was obtained by Araki~\cite{araki}.
\begin{theorem}
\label{Araki's Theorem}
Let $\Theta$ be an involution on a Hilbert space $\scrH$ and let $Q\in \mathscr{B}(\scrH)$ be such that
\begin{gather}\label{Araki's cond}
0\leq Q=Q^*\leq 1,\qquad
Q+\Theta Q\Theta = \text{Id}_\scrH.
\end{gather}
Then the identity
\begin{gather}\label{Araki's cond I}
\omega\big(\Xi(\Theta \Phi_1),\Xi(\Phi_2)\big)=\scalar{\Phi_1}{Q\Phi_2}_\oplus\qquad\forall \Phi_1,\Phi_2\in \scrH,
\end{gather}
defines a quasifree state on the CAR algebra $\mathcal{A}$ generated by elements of $\scrH$. Conversely, for every quasifree states $ \omega$ on $\mathcal{A}$ there exists a bounded linear operator  $Q$ on $\scrH$ such that  \eqref{Araki's cond} and \eqref{Araki's cond I} are satisfied.
\end{theorem}

\begin{definition} 
We call {\em basis projection} any operator $\Pi$ on $\scrH$ satisfying conditions \eqref{Araki's cond}. 
\end{definition}

Let now $\scrH^\oplus$ the Hilbert space defined in~\eqref{Hilb sol} and let $\Theta$ the involution defined in~\eqref{invol}. As an immediate corollary, we observe that to construct a basis projection $\Pi$ for $\scrH^\oplus$ it is enough to construct an orthonormal projector $\Pi$ on the pre-Hilbert space $\sol_{sc}(\overline\Dir_\M)$.

\begin{corollary}\label{Araki's corollary}
Let $\Upsilon$ be the adjunction map defined as in Section~\ref{sec:Dirac}
and $\Pi$ a orthonormal projector on the pre-Hilbert space $\sol_{sc}(\overline\Dir_\M)$. Then the operator
$P := \Pi \oplus (\text{Id}_\scrH - \Upsilon \Pi\Upsilon^{-1} )$
is a basis projection on $\scrH^\oplus$. 
\end{corollary}

Corollary~\ref{Araki's corollary} is the linking point between Hadamard states and the results from the rest of the paper. Indeed,  we observe that the pseudodifferential projections from Theorem~\ref{theorem pseudodifferential projections} can be modified,  by adding infinitely smoothing operators,  in such a way that conditions \eqref{pseudodifferential projections equation 1}, \eqref{pseudodifferential projections equation 2},  and \eqref{pseudodifferential projections equation 4}--\eqref{pseudodifferential projections equation 6} are satisfied exactly, not only modulo $\Psi^{-\infty}$. To see this, consider the operator\footnote{Recall that we assume pseudodifferential projections to be self-adjoint.}
\begin{equation}
\label{opeartor exact projections}
\mathcal{P}(t):=P_+(t)+2P_-(t).
\end{equation}
The operator \eqref{opeartor exact projections} is a self-adjoint pseudodifferential operator of order zero. It has two points of essential spectrum, $1$ and $2$,  plus possibly isolated eigenvalues of finite multiplicity.
Let $\Gamma_j(t)$, $j=1,2$, be positively oriented contours in the complex plane chosen in such a way that
\begin{itemize}
\item $\Gamma_j(t)$ encircles the point $j$,
\item $\Gamma_j(t)$ does not intersect any isolated eigenvalue of finite multiplicity of $\mathcal{P}(t)$ and
\item $\Gamma_1(t) \cup \Gamma_2(t)$ encircles the whole spectrum of $\mathcal{P}(t)$.
\end{itemize}
Then, by the elementary properties of Riesz projections,  the operators
\begin{equation*}
\widetilde{P}_+(t):=\frac{1}{2\pi i}\int_{\Gamma_1}(\mathcal{P}(t)-\lambda\, \mathrm{Id})^{-1}\,d\lambda,
\end{equation*}
\begin{equation*}
\widetilde{P}_-(t):=\frac{1}{2\pi i}\int_{\Gamma_2}(\mathcal{P}(t)-\lambda\, \mathrm{Id})^{-1}\,d\lambda,
\end{equation*}
satisfy conditions \eqref{pseudodifferential projections equation 1}, \eqref{pseudodifferential projections equation 2},  and \eqref{pseudodifferential projections equation 4}--\eqref{pseudodifferential projections equation 6} exactly,
and
\begin{equation*}
\widetilde{P}_\pm(t)=P_\pm(t) \mod \Psi^{-\infty}.
\end{equation*}
The latter implies, in particular, that $\widetilde{P}_\pm$ satisfy \eqref{pseudodifferential projections equation 3}. Condition \eqref{pseudodifferential projections equation 3} cannot, in general, be satisfied exactly by adding infinitely smoothing corrections.

\begin{definition}
\label{definition left and right projections}
We define $P_{\pm,L}(t)$ and $P_{\pm,R}(t)$ to be families of pseudodifferential operators given by Theorem~\ref{theorem pseudodifferential projections} for $A_t=\overline \Dir_t$ and  $A_t=-\overline  \Dir_t$, respectively,  (compare \eqref{eq:Cauchy probl U} and \eqref{eq:Cauchy probl U_R}) modified in such a way that they satisfy conditions \eqref{pseudodifferential projections equation 1}, \eqref{pseudodifferential projections equation 2},  and \eqref{pseudodifferential projections equation 4}--\eqref{pseudodifferential projections equation 6} exactly and we set 
$$\Pi_\pm(t):=\begin{pmatrix}
P_{\pm,L}(t) & 0\\
0 & P_{\pm,R}(t)
\end{pmatrix}\,.$$
\end{definition}

We finally have at our disposal all the ingredients to prove Theorem~\ref{thm:Hadamard}.

\begin{proof}[Proof of Theorem~\ref{thm:Hadamard}.]
On account of the properties of the operators $P_{\pm,L}(t)$,  $P_{\pm,R}(t)$ it follows that $\Pi_\pm$ are pseudodifferential orthonormal projection, whose full symbol is explicitly determined by the full symbols of $P_{\pm,L}(t)$ and $P_{\pm,R}(t)$, in turn obtained via the algorithm from Subsection~\ref{Algorithmic construction}.  Formula~\eqref{eq:Cauchy surf cov} is then obtained by applying Corollary~\ref{Araki's corollary}. That $\tilde\lambda^\pm$ are of Hadamard form follows at once from Proposition~\ref{proposition stationary points of the Levi-Civita phase functions} and~\cite[Proposition 3.8]{gerardDirac}. 
\end{proof}

\begin{remark}
Different ways of distributing the isolated eigenvalues of $\mathcal{P}$ between $\Gamma_1$ and $\Gamma_2$ yield different orthogonal projections and, ultimately, different Hadamard states.
\end{remark}

\subsection{Construction of Feynman propagator}
\label{sec:proof2}

\begin{proof}[Proof of Theorem~\ref{thm:Feynman}]
Let $U^{(\pm)}:=U^{(\pm)}_L\oplus U^{(\pm)}_R$ be the positive/negative  Dirac propagators for the reduced  Dirac operator.  The combination of Proposition~\ref{proposition stationary points of the Levi-Civita phase functions} with~\cite[Proposition 3.8]{gerardDirac} implies that the wavefront set of  $U^{(\pm)}$ is of the Hadamard form.  Then it is not hard to check that  
\[
\G_F(t,s):= U^{(+)}(t,s) - \theta(s-t) U(t,s)
\]
is a Feynman propagator,  since $-\theta(s-t) U(t,s)$ is the `time kernel' of the advanced Green operator. 
\end{proof}

\color{black}

\end{document}